\documentclass[11pt]{amsart} 

\usepackage{amsmath,amssymb,amsthm,amsfonts,verbatim}
\usepackage{enumerate}
\usepackage{color}
\usepackage{graphicx}

\theoremstyle{plain}
\newtheorem{thm}{Theorem}

\newtheorem{theorem}{Theorem}[section]

\newtheorem{proposition}[theorem]{Proposition}

\newtheorem{lemma}[theorem]{Lemma}

\newtheorem{corollary}[theorem]{Corollary}
\theoremstyle{definition}
\newtheorem{definition}[theorem]{Definition}
\newtheorem{remark}[theorem]{Remark}

\newcommand{\nc}{\newcommand}
\nc{\dmo}{\DeclareMathOperator}

\nc{\Q}{\mathbb{Q}}
\nc{\R}{\mathbb{R}}
\nc{\RP}{\mathbb{RP}^1}
\nc{\Z}{\mathbb{Z}}
\renewcommand{\P}{\mathbb{P}}
\nc{\ZZ}{\mathbb{Z}[\mathbb{Z}]}
\nc{\C}{\mathbb{C}}
\nc{\cS}{\mathcal{S}}
\nc{\iso}{\cong}
\dmo{\Mod}{Mod}
\dmo{\Ig}{\mathcal{I}_g}
\dmo{\Span}{span}
\dmo{\Diff}{Diff}
\dmo{\Homeo}{Homeo}
\dmo{\dist}{dist}
\dmo\BDiff{BDiff}
\dmo\SO{SO}
\dmo\slide{sl}
\dmo\im{im}
\dmo\id{id}
\dmo\Fix{Fix}
\dmo\Out{Out}

\dmo\Sp{Sp}
\dmo\GL{GL}
\dmo\U{U}
\dmo\SU{SU}
\usepackage{overpic,float}

\renewcommand{\epsilon}{\varepsilon}
\nc{\coloneq}{\mathrel{\mathop:}\mkern-1.2mu=}
\nc{\margin}[1]{\marginpar{\scriptsize #1}}
\nc{\para}[1]{\bigskip\noindent\textbf{#1}}

\begin{document}
\title{Exponential Torsion Growth for random 3-manifolds}
\date{March 10, 2017}
\thanks{H. Baik, I. Gekhtman, U. Hamenst\"adt supported by ERC Advanced
Grant ``Moduli'',\\
B. Petri supported by the Max Planck Institut Bonn.}
\author[]{Hyungryul Baik, David Bauer, Ilya Gekhtman, Ursula Hamenst{\"a}dt,
Sebastian Hensel, Thorben Kastenholz, Bram Petri, Daniel Valenzuela}
\maketitle

\begin{abstract}
  We show that a random $3$--manifold with positive first Betti number admits a tower
  of cyclic covers with exponential torsion growth.
\end{abstract}


\section{Introduction}
\label{sec:intro}

Given a manifold $M$ and a 
\emph{tower of coverings} of $M$, i.e. a sequence
\[ \dots \to M_n \to M_{n-1} \to \dots \to M_1 \to M \]
of finite covers, one can ask about the growth of 
topological invariants for the manifolds 
in the sequence. In the case that
$M$ is a hyperbolic $3$-manifold of finite volume, the study of 
such questions 
led to interesting conjectures 
which relate the growth of invariants of the sequence to invariants of 
hyperbolic 3-space.

More concretely, in
the case of a tower of congruence covers of a closed arithmetic
hyperbolic 3-manifold, conjecturally the growth rate of the torsion
$H_1(M_i,\mathbb{Z})_{\rm tor}$ 
in the first homology group coincides with the $\ell^2$-torsion of
$\mathbb{H}^3$, which equals $\frac{1}{6\pi}$ (Conjecture 1.4 
in \cite{BV} proposes a slightly weaker statement).

Much earlier, torsion homology growth was studied for towers of abelian covers
of knot complements. First results on the relation of this growth rate
to the (logarithmic) Mahler measure of the Alexander polynomial of the 
knot or link can be found in
\cite{Ri90} and \cite{GS91}. Equality of this
growth rate and the Mahler measure of the Alexander polynomial are due to 
Silver and Williams
\cite{SW02a}, and extensions of these results and an interpretation 
in the context of $\ell^2$-invariants can be found in  
\cite{SW02b} as well as in the more recent papers
\cite{BV,R,Le}.



\smallskip 
As it became apparent in recent years, the existence of towers of covers
with exponential torsion homology growth should be abundant for 
3-manifolds. The recent work \cite{BGS17} explains that however,
such towers do not exist for manifolds of higher dimension. 
The goal of this paper is to study the existence of towers of
cyclic covers with exponential torsion growth for random $3$--manifolds in a sense that we make precise next.

Any closed 3-manifold $M$ admits a \emph{Heegaard decomposition}.
This means that $M$ can be obtained by gluing two \emph{handlebodies}
of some genus $g\geq 0$ with a diffeomorphism of their boundaries.
The smallest genus of a handlebody which gives rise to 
$M$ in this way is called the \emph{Heegaard genus} of $M$.

For a fixed
base identification, the manifold $M$ only depends
on the element in the \emph{mapping class group}
${\rm Mod}(S_g)$ 
of the
boundary surface $S_g$ of the handlebody 
defined by the gluing diffeomorphism.
We denote by $N_\varphi$ the closed
$3$-manifold defined by the gluing map $\varphi\in {\rm Mod}(S_g)$.
Thus topological properties of closed $3$-manifolds $N_\varphi$ are 
directly related to properties of the mapping class $\varphi$.

This viewpoint was used by Dunfield and Thurston 
\cite{DT} to define the 
notion of a \emph{random 3-manifold} using a 
random walk on the mapping class group.
Embarking from \cite{DT}, 
the purpose of this work is to study cyclic covers\footnote{We always assume that cyclic covers are regular.} 
of random hyperbolic 3-manifolds with positive first
Betti number. 

\smallskip Let $\Ig$ be the \emph{Torelli subgroup} of
${\rm Mod}(S_g)$, i.e. the subgroup formed by all those mapping
classes which act trivially on $H_1(S_g;\mathbb{Z})$. For $g \geq 3$
this is a finitely generated group.  We use the following model for
random $3$--manifolds with large Betti number, which is inspired by
(but slightly different from) the Dunfield--Thurston model.  Take any
probability measure $\mu$ on ${\mathcal I}_g$ whose support equals a
finite set which generates ${\mathcal I}_g$ as a semigroup. Such a
$\mu$ defines a random walk on ${\mathcal I}_g$. We say that a
property ${\mathcal P}$ holds for a random $3$-manifold of Heegaard
genus $g$ and maximal homology rank if the following holds: the
proportion of $3$-manifolds with ${\mathcal P}$ which are defined by a
gluing with an element of the $n$-th step of the walk tends to one as
$n\to \infty$, independently of $\mu$.  To motivate this model, note
that any $3$--manifold $M$ with Heegaard genus $g$ and first Betti
number $b_1(M)=g$ is obtained as $N_\varphi$ for some
$\varphi \in \Ig$ (compare Section \ref{sec:setup}) and by \cite{DT}
(see also \cite{LMW14}), the Heegaard genus of a $3$-manifold
obtained from a random gluing in this sense is $g$. Furthermore, by
a theorem by Maher \cite{Ma10}, a random manifold with Heegaard genus $g$
and maximal homology rank is indeed hyperbolic.

\begin{thm}\label{maintheorem1}
A random $3$-manifold of Heegaard genus $g\geq 3$ with maximal homology rank has
a tower of cyclic covers with exponential torsion homology growth.
\end{thm}

A precise version of this result is 
Theorem~\ref{thm:main-theorem-plain} in Section~\ref{sec:finishing}.
We do not discuss the rate of convergence although we believe that it can be derived from careful analysis of Benoist and Quint's work on random walks on reductive groups \cite{BQ,BQ2} (see also Section 7.6 of \cite{K} and \cite{LMW14}).

A result analogous to Theorem~\ref{maintheorem1} remains true for
random $3$--manifolds with positive first Betti number (as opposed to
maximal) by considering random walks not on the Torelli group $\Ig$,
but on a \emph{homology stabiliser} which is defined to be the
subgroup of ${\rm Mod}(S)$ of all mapping classes preserving some
fixed homology class.

Theorem~\ref{maintheorem1}
mainly is a result about the Torelli subgroup of the mapping class
group. It does not rely on any information on the geometry and topology of hyperbolic 
3-manifolds. Very recently, such topological/geometric tools were used by Liu~\cite{L} 
to construct for an arbitrary closed hyperbolic 3-manifold $M$
a tower of 
covers of $M$ with exponential torsion homology growth.
These covers  are however in general not regular. 
The existence of a finite cover $M^\prime$ of $M$ which admits a 
tower of cyclic covers with exponential torsion homology
growth follows from 
Liu and Sun's beautiful virtual domination theorem~\cite{S15a,S15b,LiS16}.

\smallskip
The methods used in the proof of Theorem~\ref{maintheorem1} can also
be used to show that the first Betti number
of a random hyperbolic three-manifold $N_\varphi$  
does not increase by passing
to a finite Abelian cover of fixed degree.
\begin{thm}\label{maintheorem2}
  Fix a natural number $d > 0$.
A random $3$-manifold of Heegaard genus $g\geq 3$ with maximal homology rank has
no Abelian cover of degree $\leq d$ with Betti number $>g$.
\end{thm}
A precise version of this result is 
Theorem~\ref{thm:dt} in Section~\ref{sec:finishing}.
Theorem~9.1 in \cite{DT} shows that for a fixed number $k>0$, a random
hyperbolic 3-manifold of Heegaard genus two does not admit an Abelian
cover of degree at most $k$ with positive first Betti number, where
random refers to a random walk on the entire mapping class group.  Our
methods can be adapted to extend the result of \cite{DT} to
arbitrary genus\footnote{In \cite{Ri}, Rivin claims this conclusion
  for solvable covers of random 3-manifolds of arbitrary Heegaard
  genus.  However, his argument seems incomplete, although we believe
  that it can be completed in the Abelian case.}, 
as outlined in the last section.  We chose instead to
focus on the version stated in Theorem~\ref{maintheorem2} as its formulation is
closer to the formulation of Theorem~\ref{maintheorem1}.

As a final application of our methods, we also show that for a random walk on the full mapping class group, the order of $H_1(N_\varphi,\Z)_{\mathrm{tors}}$ grows exponentially in the number of steps of the random walk (Theorem \ref{thm:kowalski}), answering a question of Kowalski \cite{K}.

\smallskip
The proof of Theorem~\ref{maintheorem1} relies on 
the relation between the growth rate of torsion in 
the homology for a tower of cyclic covers of the three-manifold
$N_\varphi$ 
and the Mahler measure of the Alexander polynomial $\Delta$
of the corresponding infinite cyclic covering.

Given $\varphi\in \Ig$, 
there is an infinite cyclic covering of $N_\varphi$  induced by an infinite cyclic 
covering $\widetilde{S}$ of the surface $S$. 
In Section~\ref{sec:matrix} we give an explicit description of the
homology of $\widetilde{S}$ as a $\ZZ$-module.  
We construct
a matrix $M(\varphi)$ with entries in 
the group ring $\Z[\Z]$ which describes the
action of a lift of $\varphi$ on $H_1(\widetilde{S})$. 
In Section~\ref{sec:algebra} we translate the condition that $\Delta$ has
(logarithmic) Mahler measure $0$ into a condition
that is detectable by the action of lifts of $\varphi$ on finite covers $S_q$ of the surface $S$.

Using an idea of Looijenga \cite{Lo}, we then study the action of the
Torelli group on the homology of these covers. The core result is
Proposition~\ref{prop:torelli-image-dense}. It shows that lifts of elements in
the Torelli group generate a dense subgroup of an algebraic group
which is closely related to the automorphism group of the homology of 
some finite cyclic cover of $S$.
Using results of Benoist and Quint \cite{BQ} on random walks on
algebraic groups, we then deduce that random elements in $\Ig$ with
probability one violate the conditions implied by $\Delta$ having Mahler measure zero.

The main novelty of our approach lies in a direct translation of properties
of random walks on the Torelli group into properties of random
walks on algebraic groups. We do not use any of the recent results on
random walks on the mapping class group.  

\textbf{Acknowledgements}: This work was carried out in fall 2015
while all authors were in residence in Bonn. 
All of us thank Gregor Masbaum for useful conversations. 
Sebastian Hensel is grateful to  
Benson Farb for helpful discussions. 
We are particularly grateful to an anonymous referee who pointed out an
error in an earlier version of this paper and whose 
suggestions led moreover to a significant simplification of our argument.

\section{Covers of Surfaces and $3$--manifolds}
\label{sec:setup}
In this section we describe the setup we will use to determine and 
control the homology of covers of $3$--manifolds 
given by Heegaard splittings.
The terminology introduced in this section will be used throughout the article. 

Let $S$ be a surface of genus $g\geq 2$, identified once and for
all with the boundary $\partial V=S$ 
of a handlebody $V$.  Such a handlebody is a compact
manifold with boundary which is homeomorphic to the thickening of
an bouquet of $g$ circles embedded in $\mathbb{R}^3$. 
A \emph{meridian} of $V$ is an essential simple closed
curve on $S$ which bounds a disk in $V$.

Let
$\alpha_1,\ldots,\alpha_g$ be a set of (oriented) simple closed curves in $S$ 
which form a \emph{cut system} for $V$. This means that the curves
 $\alpha_i$ are pairwise 
disjoint meridians for $V$ whose complement
$S-\cup_i\alpha_i$ is connected. In particular, the $\alpha_i$ are pairwise non-homologous and
all non-separating.
Then 
\[ L = \ker(H_1(S;\Z) \to H_1(V;\Z)) = \Span_{\Z}\{[\alpha_1], \ldots, [\alpha_g]\},\]
and $L$ 
is a Lagrangian subspace of 
$H_1(S,\mathbb{R})$ with respect to the algebraic intersection pairing 
\[(\cdot,\cdot):H_1(S,\mathbb{R})\times H_1(S,\mathbb{R})\to \mathbb{R}\] 
on homology.

Let $\beta_1,\ldots,\beta_g$ be a set of simple closed curves on $S$ dual to
the cut system $\{\alpha_i\}$. This means that the curves $\beta_i$ are 
pairwise disjoint  (and transverse to the curves $\alpha_j$ for some
smooth structure), and 
$\#(\alpha_i\cap\beta_j) = \delta_{ij}$. We assume that the
$\alpha_i,\beta_j$ are oriented so that $(\alpha_i,\beta_j)\geq 0$ for all $i,j$.
The $\alpha_i,\beta_i$ project to a symplectic basis
$a_1,\ldots,a_g,b_1,\ldots,b_g$ 
of $H_1(S;\Z)$ (here $a_i=[\alpha_i]$ for the above notation). 

\smallskip
Given any mapping class $\varphi\in \mathrm{Mod}(S)$ of $S$, we denote by $N_\varphi$ the $3$-manifold
given by the Heegaard splitting defined by $\varphi$:
\[ N_\varphi = V \cup_\varphi V \]
With our convention, the identity mapping class gives rise to the manifold
$N_{\mathrm{id}} = S^2\times S^1\sharp \dots \sharp S^2\times  S^1$ ($g$ copies).
We have (see the beginning of Section 8 of \cite{DT}). 

\begin{lemma}\label{lem:heegaard-homology}
\[ H_1(N_\varphi;\Z) = H_1(S;\Z) \left/ \langle L, \varphi_* L \rangle \right. \]
where $\varphi_*$ denotes the induced map of $\varphi$ on homology.  
\end{lemma}
The following lemma relates the first 
Betti number $b_1(N_\varphi)$ of
$N_\varphi$ to information on the gluing map $\varphi$. For its statement, we need
to introduce certain subgroups of $\mathrm{Mod}(S)$. 
The \emph{Torelli group}  $\Ig$ is the group of all mapping classes acting trivially
on $H_1(S;\Z)$. The \emph{handlebody group} is the subgroup 
$\mathcal{H}_g$ of ${\rm Mod}(S)$ of 
those mapping classes which can be represented by diffeomorphisms of $S$ extending 
to $V$. Finally, given homology
classes $a,b\in H_1(S;\Z)$ we denote the \emph{homology stabiliser group} by
\[ HS(a,b) = \{ f \in \mathrm{Mod}(S) | f_*(a) = a, f_*(b) = b \} \]
Note that $\Ig < HS(a,b)$ for any $a,b$.
\begin{lemma}\label{lem:betti-number-criterion}
  \begin{enumerate}[a)]
  \item $b_1(N_\varphi) \leq g$ with equality if and only
    if $\varphi = \psi_1 \psi_2$ where
    $\psi_1 \in \Ig, \psi_2 \in \mathcal{H}_g$.
  \item $b_1(N_\varphi) \geq 1$ if and only if
    there are     $a\in L, b\in H_1(S;\Z), (a,b) = 1$ so that 
    $\varphi = \psi_1 \psi_2$ with
    $\psi_1 \in HS(a,b), \psi_2 \in \mathcal{H}_g$.
  \end{enumerate}
\end{lemma}
\begin{proof} 
  Lemma~\ref{lem:heegaard-homology} shows that the first Betti number
  of $N_\varphi$ is at most $g$. The same simple observation which
  leads to Lemma \ref{lem:heegaard-homology} (see the discussion in
  Section 8 of \cite{DT}) also yields that the conditions in a), b)
  are sufficient for the Betti number bound. We proceed to show necessity.

  \begin{enumerate}[a)]
  \item By Lemma~\ref{lem:heegaard-homology}, if $b_1(N_\varphi)=g$ 
    then $\varphi_* L = L$. In other words, the matrix describing $\varphi_*$ 
    with respect to the symplectic basis introduced above has the form
    \[
    \begin{pmatrix}
      A & B \\ 0 & C
    \end{pmatrix}
    \]
    Now any symplectic matrix of such a form
    is induced by an element of the handlebody group \cite{Hir}. 
    The claim
    follows.
  \item By Lemma~\ref{lem:heegaard-homology}, if $b_1(N_\varphi)\geq 1$ 
    then there is some $0 \neq v \in L \cap \varphi_*L$. Since $\varphi_*$ is an automorphism
    of $H_1(S,\mathbb{Z})$ we may assume that $v$ is \emph{primitive} (which 
    is equivalent to stating that $v$ can be represented by a simple closed curve, see
    \cite[Proposition~6.2]{FM}). 
    Since $\mathcal{H}_g$ acts transitively on the set of primitive
    vectors in $L$, by multiplying $\varphi$ from the right by an element in
    $\mathcal{H}_g$ we may assume that $\varphi_*(v) = v$. Using the description of
    the image of the handlebody group in $\mathrm{Sp}(2g, \mathbb{Z})$ given above
    \cite{Hir}, it follows that the stabiliser in $\mathcal{H}_g$ of an element 
    $v\in L$ acts
    transitively on the set of primitive vectors $w\in H_1(S;\Z)$ with $(v,w) = 1$. 
    The claim follows.
  \end{enumerate}
\end{proof}

Next we discuss how covers of $S$ give rise to covers of $N_\varphi$.
The following easy lemma can also be found in \cite{DT}\footnote{In the terminology of \cite{DT}: a map $\sigma$ induces a cover of $N_\varphi$ if and
only if 
$\sigma$ extends over $V$ and $\varphi\cdot \sigma$ extends over $V$.}. 
\begin{lemma}\label{lem:criterion-covers}
  Let $\sigma:\pi_1(S) \to G$ be a surjection onto a group $G$. Then
  $\sigma$ factors through a map $\pi_1(N_\varphi)\to G$
  if and only if \[K = \ker(\pi_1(S)\to \pi_1(V)) \subset \ker(\sigma)
  \text{ and }\varphi_*K\subset\ker(\sigma).\]
\end{lemma}
\begin{proof}
  This is an immediate consequence of the fact that
  \[ \pi_1(N_\varphi) = \pi_1(S) / \langle\langle K, \varphi_*K \rangle\rangle \]
 where $\langle\langle K,\varphi_*K\rangle\rangle$ denotes the normal closure  
 of the subgroup of $\pi_1(S)$ generated by $K,\varphi_*K$.
 The former statement can e.g. be derived from the 
theorem of Seifert--van Kampen.
\end{proof}
In particular, we have the following.
\begin{corollary}\label{cor:covers-for-torelli-gluings}
  Let $\sigma:\pi_1(S)\to G$ be a surjection onto 
  an Abelian group $G$ so that $K = \ker(\pi_1(S)\to \pi_1(V)) \subset \ker(\sigma)$.
  Denote by $S'$ the cover of $S$ defined by $\sigma$. 
  Let $\varphi\in\Ig$ be arbitrary. Then:
  \begin{enumerate}[i)]
  \item $S' = \partial V'$ for a cover $V'$ of $V$,
    and the action of the deck group $G$ on $S'$ extends to the
    action of the deck group of $V'\to V$.
  \item $\sigma$ factors through a map $\sigma_\varphi:\pi_1(N_\varphi)\to G$.
  \item The cover $\widetilde{N_\varphi}\to N_\varphi$ defined by $\sigma_\varphi$ is homeomorphic
    to $N_{\widetilde{\varphi}}$, where $\widetilde{\varphi}$ is any lift of
    $\varphi$ to $S'$.
  \end{enumerate}
  The same remains true if $\varphi\in HS(a,b)$, assuming that $\sigma:\pi_1(S)\to G$ is
  defined by algebraic intersection number (possibly mod $q>0$) with $a$.
\end{corollary}
\begin{proof}
To show the first assertion, 
  suppose that $K\subset\ker(\sigma)$ and let 
  $V' \to V$ be the cover of the handlebody $V$ 
  whose fundamental group is 
  the image
  $\ker(\sigma) / K$ in $\pi_1(V)$ 
of the subgroup $\ker(\sigma)<\pi_1(S)$.
  Note first that
  \[ G = \pi_1(S) / \ker(\sigma) = \left(\pi_1(S) /
    K\right) / \left(\ker(\sigma) / K\right) \]
 and hence $V' \to V$ is a regular cover
  with deck group $G$. The induced cover $\partial V' \to \partial V = S$ has
  fundamental group exactly $\ker(\sigma)$, and therefore it is equal to $S'\to S$.

The second statement is immediate from Lemma  
\ref{lem:criterion-covers} and the fact that $\sigma$ factors
through a homomorphism $H_1(S,\mathbb{Z})\to G$. In particular,
any element of $\Ig$ lifts to $V^\prime$, and this lift commutes
with the action of $G$ which implies the third statement. Under the extra assumption 
given at the end, the same is true for $\varphi\in HS(a,b)$.
\end{proof}

We call a cover $S' \to S$ as in Corollary~\ref{cor:covers-for-torelli-gluings}
a \emph{$K$--cover}. For any $K$--cover, by part i) of that corollary, there 
is a subspace
\[ L' = \ker(H_1(S';\Z) \to H_1(V';\Z)) \]
and thus, by Lemma~\ref{lem:heegaard-homology}, we obtain

\begin{proposition}\label{prop:describe-cover-homology}
  With notation as above, we have 
  \[ H_1(\widetilde{N_\varphi}; \Z) = H_1(N_{\widetilde{\varphi}}; \Z) = H_1(S';\Z) \left/ \left\langle L', \widetilde{\varphi}_*L'\right\rangle\right. \]
where $\widetilde{\varphi}$ is any lift of $\varphi$ to $S'$.
\end{proposition}

We will also need the following version 
which is useful to compare Betti numbers. 
To this end, define
\[ E = \ker( H_1(S';\Z) \to H_1(S;\Z) ) \]
and let $L'_E = L' \cap E$. By transfer, we have
\[  H_1(S';\Q) = H_1(S;\Q) \oplus (E\otimes \Q). \]
Furthermore, $E$ and this decomposition is preserved by $\widetilde{\varphi}_*$.
This, together with Proposition~\ref{prop:describe-cover-homology} yields the following
useful characterisation.
\begin{proposition}\label{prop:describe-cover-homology-nontriv}
  With notations as above, we have 
  \[ H_1(\widetilde{N_\varphi}; \Q) = H_1(N_\varphi; \Q) \oplus 
\left( E\otimes\Q \left/ \left\langle L'_E\otimes\Q, 
\widetilde{\varphi}_*L_E'\otimes\Q\right\rangle\right. \right) \]
  In particular, $b_1(\widetilde{N_\varphi}) > b_1(N_\varphi)$ if and only if
  \[ L'_E + \widetilde{\varphi}_*L_E' \subsetneq E \]
  is not a lattice.
\end{proposition}

In the sequel, two special kinds of $K$--covers will be particularly important.
Namely, given any primitive vector $a \in L$, the kernel of the map which 
associates to an element $\alpha\in \pi_1(S)$ the algebraic intersection 
number with $a$ of the homology class defined by $\alpha$ 
(resp. its algebraic intersection number with $a$ mod $q$
for $q\in\mathbb{N}$) defines an infinite
cyclic cover $S_\infty \to S$ (resp. a cyclic cover $S_q \to S$ of order $q$). 
Informally, we call such a cover the cover defined by algebraic intersection
number with $a$.

By Lemma~\ref{lem:betti-number-criterion}, if $N_\varphi$ has Betti number at least
$1$, then $a$ may be chosen such that these covers induce covers of the
$3$--manifold $N_\varphi$. In the sequel we always do so.

\section{Cyclic Covers}
\label{sec:matrix}

Recall from Section \ref{sec:setup} the choice 
$a_1,\dots,a_g,b_1,\dots,b_g$ of a symplectic basis of 
$H_1(S,\mathbb{Z})$.  The classes $a_1,\ldots, a_g$ generate the kernel $L$ of
the map $H_1(S;\Z) \to H_1(V;\Z)$ induced by the inclusion $S = \partial V \to V$.
Let $S_\infty\to S$ be the infinite cyclic $K$-cover of $S$ defined by
algebraic intersection number with $a_g$. There is a corresponding
infinite cyclic cover $V_\infty \to V$. 
Since linear functionals on $H_1(S,\mathbb{Z})$ defined by algebraic 
intersection with non-separating simple closed curves generate $H^1(S;\Z)$ and the
mapping class group acts transitively on such curves, the results in this section in
fact hold true for any cyclic $K$--cover $S_\infty\to S$. We will restrict to the case of
intersection with $a_g$ for clarity.

\subsection{Homology of the Infinite Cyclic Cover}
The goal of this subsection is to 
give a fairly explicit (but non-canonical) description
of the first homology of $S_\infty$ as a module over the
group ring $\mathbb{Z}[\mathbb{Z}]$ of the deck group 
$\mathbb{Z}$ of $S_\infty$.

Denote by $Y \subset S$ the complementary subsurface of the simple closed
curve $\alpha_g$ on $S$. Choose 
a preferred lift $\widetilde{Y}$ of $Y$ to $S_\infty$, 
i.e. $\tilde Y\subset S_\infty$
is a connected subsurface with boundary which is mapped by the covering map 
$S_\infty\to S$ homeomorphically onto $Y$.

For $i,j\leq g-1$ the curves $\alpha_i,\beta_j$ (defining the homology classes
$a_i,b_j$) admit
unique lifts
\[\widetilde{\alpha}_1,\ldots,\widetilde{\alpha}_{g-1},\widetilde{\beta}_1,\ldots,\widetilde{\beta}_{g-1}\]
to $\widetilde{Y}$. Let $\widetilde{\alpha}_g$ be the lift of
$\alpha_g$ contained in the closure of $\widetilde{Y}$ whose
orientation agrees with the boundary orientation of
$\widetilde{Y}$. Note that this makes sense since the orientation of
$S$ induces an orientation of $\tilde Y$ and since $\alpha_g$ is an
oriented curve\footnote{This just serves to fix a specific lift, we
  could take any.}. Denote by $\widetilde{a}_i, \widetilde{b}_j$
$(1\leq i\leq g,1\leq j\leq g-1)$ the homology classes in
$H_1(S_\infty;\Z)$ of the curves $\widetilde
\alpha_i,\widetilde\beta_j$. We have an isomorphism $H_1(Y;\Z) \to
H_1(\widetilde{Y};\Z)$ which sends $a_i, b_j$ to $\widetilde{a}_i, \widetilde{b}_j$.

As usual, we denote by $\mathbb{Z}[G]$ 
the integral group ring of a group $G$.
Then $\mathbb{Z}[\mathbb{Z}]$ has an obvious
identification with the ring $\mathbb{Z}[t,t^{-1}]$ of integral
Laurent polynomials. If $H$ is a $\Z$--module, we write $H[t, t^{-1}]$ to mean
$H \otimes_\Z \ZZ$.
We also choose $\tau$ a generator of the deck group of $S_\infty$.

\begin{lemma}\label{lem:generators-h1-sinfty}
  The map
  \[ H_1(Y;\Z)[t, t^{-1}] \oplus \mathbb{Z} \cong \mathbb{Z}[\mathbb{Z}]^{2g-2} \oplus \mathbb{Z} \to H_1(S_\infty; \mathbb{Z}) \]
  induced by sending $a_i$ to $\widetilde{a}_i$, $b_i$ to $\widetilde{b}_i$ for $i<g$
  and the generator of the trivial $\mathbb{Z}[\mathbb{Z}]$--module $\mathbb{Z}$ to $\widetilde{a}_g$ 
  is an isomorphism of $\mathbb{Z}[\mathbb{Z}]$--modules.

  Furthermore, the image of $L[t, t^{-1}] \oplus
  \mathbb{Z}$ under this map is exactly \[\ker(H_1(S_\infty;\Z) \to H_1(V_\infty;\Z)).\]
\end{lemma}
\begin{proof}
  Recall that $\widetilde{Y}$ is homeomorphic to $Y$, and that
  the subsurfaces $\tau^n\widetilde{Y}$ are pairwise disjoint
  ($n\in \mathbb{Z})$. The integral homology of $Y$ equals 
	$\mathbb{Z}^{2g-1}=F\oplus \mathbb{Z}\tilde a_g$ where the free $\mathbb{Z}$-module
	$F$ of rank $2g-2$ is spanned by the classes $\tilde a_i,\tilde b_j$ $(1\leq i,j\leq g-1)$.
	
	The closures of $\tau^n\widetilde{Y}$ and
  $\tau^{n+1}\widetilde{Y}$ intersect in the lift
  $\tau^n\widetilde{\alpha_g}$ of $\alpha_g$. 
	Enlarging slightly the set $\widetilde{Y}$ to a neighborhood of its closure
	allows to apply the Mayer Vietoris sequence to $\cup_i\tau^i\widetilde{Y}$ to calculate 
	the homology of $S_\infty$. We find that 
	\[H_1(S_\infty;\mathbb{Z})=\Pi(\oplus_{i\in \mathbb{Z}} H_1(\tau^i \widetilde{Y};\mathbb{Z}))\]
	where the map $\Pi$ is the identity on $\oplus_i\tau^iF$ and identifies 
	$\tau^i\tilde a_g$ with $\tau^{i+1}\tilde a_g$.
	
	Now observe that the deck group $\mathbb{Z}$ acts on both sides of this equation, and that the
	map $\Pi$ is equivariant for this action. Thus $H_1(S_\infty;\mathbb{Z})=
	\mathbb{Z}[\mathbb{Z}]^{2g-2}\oplus (\mathbb{Z}[\mathbb{Z}]/\ker(\epsilon))$
	where $\epsilon:\mathbb{Z}[\mathbb{Z}]\to \mathbb{Z}$ is the $\mathbb{Z}$-linear
	map defined by $\epsilon(g)=1$ for all $g\in \mathbb{Z}$. This shows the first assertion
	of the lemma.

  The second assertion follows from the fact that the $\alpha_i$ normally
  generate the kernel of $\pi_1(S)\to\pi_1(V)$.
\end{proof}

We define
\[ F_\infty = \mathrm{span}_{\mathbb{Z}[\mathbb{Z}]}(\widetilde{a}_1,\ldots,\widetilde{a}_{g-1},
 \widetilde{b}_{1},\ldots,\widetilde{b}_{g-1}) \] 
and note that it is a free $\mathbb{Z}[\mathbb{Z}]$-submodule
of $H_1(S_\infty;\Z)$.
We will frequently use the decomposition
\begin{equation}\label{finf} F_\infty = C_\infty \oplus D_\infty\end{equation} 
\[D_\infty = \mathrm{span}_{\mathbb{Z}[\mathbb{Z}]}\{\widetilde{a}_i, i<g\}, \quad C_\infty =
\mathrm{span}_{\mathbb{Z}[\mathbb{Z}]}\{\widetilde{b}_i, i<g\}. \] 
By the above discussion, both $C_\infty$ and $D_\infty$ are free
$\mathbb{Z}[\mathbb{Z}]$-modules of rank $g-1$.

\subsection{Matrices}
\begin{lemma}\label{lem:torelli-preserves-free}
  Suppose $\varphi\in HS(a_g,b_g)$.
  Then $\varphi$ lifts to $S_\infty$, the lift commutes with the deck group
  action,
  and the induced map $\widetilde{\varphi}_*$ on homology preserves $F_\infty$.
\end{lemma}
\begin{proof}
  The fact that $\varphi$ lifts is immediate from the 
  fact that $HS(a_g,b_g)$ preserves algebraic
  intersection number with $a_g$. The same fact also implies that any
  lift of $\varphi$ commutes with the deck group action.

  We are left with showing the last claim. Let as before
$\tau$ be a generator of the deck group of $S_\infty$.
By Lemma~\ref{lem:generators-h1-sinfty}, for any $i<g$ there are numbers
  $n^a_{k,j}, n^b_{k,j},m\in \mathbb{Z}$ so that
  \[ \widetilde{\varphi}_*( \widetilde{a}_i ) = 
\sum_{j<g, k\in\Z} n^a_{k,j}\tau^k\widetilde{a_j} + m \widetilde{a}_g + 
\sum_{j<g, k\in\Z} n^b_{k,j}\tau^k\widetilde{b}_j.\]
  Thus we have
  \[ \varphi_*(a_i) = \sum_{j<g, k\in\Z} n^a_{k,j}a_j + m a_g + \sum_{j<g, k\in\Z} n^b_{k,j}b_j \]
  Since $\varphi\in HS(a_g,b_g)$ we compute 
  \[ m = (\varphi_*(a_i), b_g) = (\varphi_*(a_i), \varphi_*(b_g)) =
  (a_i, b_g) = 0 .\]
  This implies
  $\widetilde{\varphi}_*( \widetilde{a}_i ) = \sum_{j<g, k\in\Z}
  n^a_{k,j}\tau^k\widetilde{a}_j + \sum_{j<g, k\in\Z}
  n^b_{k,j}s¸\tau^k\widetilde{b}_j \in F_\infty$.
  The case of $\widetilde{\varphi}_*( \widetilde{b}_i )$ is similar.
\end{proof}
\begin{definition}\label{def:matrices}
  Let $\varphi\in HS(a_g,b_g)$.
  \begin{enumerate}[i)]
  \item Denote by $M_\infty(\varphi)$ the $(2g-2)\times(2g-2)$--matrix
    with entries in $\ZZ$ describing the action of a lift
    $\widetilde{\varphi}_*$ of $\varphi$ on $F_\infty$, with respect
    to the basis $\widetilde{a}_1, \ldots, \widetilde{a}_{g-1},
    \widetilde{b}_1, \ldots, \widetilde{b}_{g-1}$. The matrix
    $M_\infty(\varphi)$ is well-defined up to a unit in $\ZZ$ 
    (namely, it depends on the choice of a lift of $\varphi$, any two of which differ by multiplication with some $t^k$).
  \item Denote by $B_\infty(\varphi)$ the ``bottom-left block'' of $M_\infty(\varphi)$;
    explicitly, $B_\infty(\varphi)$ is the $(g-1)\times (g-1)$-matrix 
defined by the requirement that for $j<g$
    \begin{equation}\label{Bmatrix} \widetilde{\varphi}_*(\widetilde{a}_j) = d +
    \sum_{i=1}^{g-1}(B_\infty(\varphi))_{i,j}\widetilde{b}_i 
\end{equation}
    where $d\in D_\infty$. Again, $B_\infty(\varphi)$ is only
    well-defined up to a unit in $\ZZ$ (in the same sense as above).
  \end{enumerate}
\end{definition}
Similarly, for 
$\varphi\in HS(a_g,b_g)$ the determinant 
\[ \det B_\infty(\varphi) \in \ZZ \] is defined up to a unit in
$\ZZ$. We will usually assume that the unit is chosen so that $\det B_\infty(\varphi)$ is a
polynomial.

\begin{remark}
  If $\varphi\in\Ig$ is an element of the Torelli group, then we can write $M_\infty(\varphi)=\mathrm{Id}+M'_\infty(\varphi)$ where $M'_\infty(\varphi)$ is a matrix with entries in the augmentation ideal
  $\mathfrak{a} = \ker(\ZZ\to\Z)$. This can be proved with an argument
  very much like Lemma~\ref{lem:torelli-preserves-free}.
\end{remark}

\subsection{Infinite Cyclic Covers of $3$-Manifolds}
In light of equation (\ref{Bmatrix}), the matrix $B_\infty(\varphi)$ 
defines a map $D_\infty\to C_\infty$ which describes part of the
action of $\tilde \varphi_*$.  
Its importance stems from the following immediate consequence of 
Proposition \ref{prop:describe-cover-homology}. 

\begin{proposition}\label{prop:presentation-matrix}
  Let $\varphi\in HS(a_g,b_g)$ 
  and let $(N_\varphi)_\infty \to N_\varphi$
  be the cover of $N_\varphi$ induced by $S_\infty\to S$. Then we have, as $\ZZ$--modules
  \[ H_1((N_\varphi)_\infty;\Z) = C_\infty \left/ \im
    B_\infty(\varphi) \right. \] Thus $B_\infty(\varphi)$
  is a presentation matrix of the $\ZZ$--module
  $H_1((N_\varphi)_\infty;\Z)$.
\end{proposition}

\subsection{Finite Cyclic Covers}
\label{sec:finite-covers}
As we will explain in detail in Section~\ref{sec:finishing}, exponential torsion growth
in a tower of cyclic covers of the manifold $N_\varphi$ will be governed by the logarithmic
Mahler measure of the polynomial $\det B_\infty(\varphi)$.
In Section~\ref{sec:algebra} we will describe a criterion 
for positivity of 
this Mahler measure which is detectable in finite
sub-covers of $S_\infty$. In this subsection we explain how $B_\infty(\varphi)$ 
affects the action on homology of lifts to finite sub-covers.

\smallskip
Choose $q\in\mathbb{N}_{\geq 3}$
and let $S_q$ be the $q$-fold cyclic cover of $S$ defined
by algebraic intersection number with $a_g$ mod $q$; denote by $G
\cong \mathbb{Z}/q\mathbb{Z}$ its deck group. We want to
describe the homology of $S_q$ as a $\mathbb{Z}[G]$-module 
(compare \cite[Section~4]{Lo} for a
very similar discussion) and its relation 
to the homology of $S_\infty$.

The Chevalley-Weil theorem \cite{CW} states 
that as a $G$-re\-pre\-sen\-ta\-tion space, we have
\[ H_1(S_q;\C) \iso \C[G]^{2g-2} \oplus \C^2. \]
In the case of cyclic covers, we can
obtain this more explicitly (and with integral coefficients) 
as follows: $S_q$ is
covered by $S_\infty$; denote by $\hat{\alpha}_i,\hat{\beta}_i, i<g$
the images of $\widetilde{\alpha}_i, \widetilde{\beta}_i$. 

With an
argument as in Lemma~\ref{lem:generators-h1-sinfty} one easily sees
that the homology classes $\hat{a}_i,\hat{b}_i, 1\leq i<g$ of the
curves $\hat{\alpha}_i,\hat{\beta}_i$ 
span a free $\mathbb{Z}[G]$-submodule $\Z[G]^{2g-2}$ of $H_1(S_q;\Z)$.
Choose lifts 
$\hat{\alpha}_g$ of $\alpha_g$ 
and  $\hat{\beta}_g$ of $\beta_g$. It is
easy to see that their homology classes $\hat a_g,\hat b_g$  
span a trivial $\mathbb{Z}[G]$-module $\Z^2$. 

Thus, 
\begin{equation}\label{Fq}
F_q = \mathrm{span}_{\Z[G]}\{\hat{a}_i,\hat{b}_i, i<g\} \end{equation}
is a free $\mathbb{Z}[G]$-submodule of $H_1(S_q;\Z)$ which decomposes as
\[ F_q = C_q\oplus D_q, \quad D_q = \mathrm{span}_{\Z[G]}\{\hat{a}_i,
i<g\}, C_q = \mathrm{span}_{\Z[G]}\{\hat{b}_i, i<g\} \]

As in Lemma~\ref{lem:torelli-preserves-free}, any lift $\hat{\varphi}$
of an element $\varphi\in HS(a_g,b_g)$ respects the decomposition
$H_1(S_q;\Z)\iso F_q\oplus \Z^2$. Hence, we may define matrices
$M_q(\varphi), B_q(\varphi)$ as in Definition~\ref{def:matrices} (replacing
all lifts $\widetilde{\cdot}$ to $S_\infty$ by the corresponding lifts $\hat{\cdot}$
to $S_q$).

Summarising, we have the following.
\begin{lemma}\label{lem:covering-map-compatibility}
  With the identifications as above, the covering map $S_\infty \to
  S_q$ induces on homology a map
  \[ H_1(S_\infty;\Z) \iso \ZZ^{2g-2} \oplus \Z \to \Z[G]^{2g-2}
  \oplus \Z^2 \iso H_1(S_q;\Z) \] inducing via the quotient homomorphism  
  $\mathbb{Z}\to \mathbb{Z}/q\mathbb{Z}$ 
  a map between the free modules 
  \[ F_\infty = C_\infty \oplus D_\infty \to C_q\oplus D_q = F_q\]
 which respects the direct sum
  decompositions.  Furthermore, $M_q(\varphi),
  B_q(\varphi)$ are the images of $M_\infty(\varphi),
  B_\infty(\varphi)$ under coordinate-wise application of the ring morphism 
  $\ZZ\to\Z[G]$ induced by $\Z\to\Z/q\Z$.
\end{lemma}

\section{The Torelli representation for a finite cover}
\label{sec:c-side}
The purpose of this section is to show that for a finite cover
as discussed in Section~\ref{sec:finite-covers}, the matrices
$M_q(\varphi), \varphi\in\Ig$ map to a Zariski dense subset 
of some suitable algebraic group. This will be the main ingredient in Section~\ref{sec:finishing}
that allows to use the results by Benoist and Quint on random walks.
We can restrict here to the case of $\varphi \in \Ig$ for simplicity; as
$\Ig < HS(a_g,b_g)$ the image of $HS(a_g,b_g)$ will then also be Zariski dense.

\smallskip
The arguments rest on the results in \cite[Section~4]{Lo}. 
We begin with explaining these results in the form we need.
%

\subsection{Skew-Hermitian modules}\label{sec:skew-hermitian-modules}
We require a few classical results on automorphism groups of
skew-Hermitian modules; we only summarise the most important points
here, details can be found in \cite{K-Theory}.

Let $\mathcal{R}$ be a ring
with an involution $\overline{\cdot}:\mathcal{R}\to\mathcal{R}$ 
which is a homomorphism with respect to the additive structure of $\mathcal{R}$.
Let ${\mathcal M}$ be a module over $\mathcal{R}$.
A \emph{skew-Hermitian form} on $\mathcal{M}$ is a non-degenerate sesqui-linear form $\langle \cdot,\cdot\rangle :\mathcal{M}\times\mathcal{M} \to \mathcal{R}$ such that
\[ \langle x,y\rangle = -\overline{\langle y,x\rangle}\]
for all $x,y\in\mathcal{M}$.

Given a module $(\mathcal{M},\langle ,\rangle)$ with such a form, 
we will denote the group of automorphisms that preserve the form by
\[\U(\mathcal{M}) = 
\left\{ \alpha\in\mathrm{Aut}(\mathcal{M}) |\; 
\langle \alpha x,\alpha y \rangle = \langle x,y\rangle\right\}. \]
It is well known that in the special case where $\mathcal{M}=\C^{2n}$ for some $n\in\mathbb{N}$ and the involution is given by complex conjugation, we can find a basis
\[\{x_i,y_i\}_{i=1}^n \]
of $\C^{2n}$ such that
\[\langle x_i,y_j\rangle =
\delta_{ij}\;\;\text{and}\;\;\langle x_i,x_j\rangle =\langle y_i,y_j\rangle =0\]
for all $1\leq i,j\leq n$. In this case we will write
\[\U(\mathcal{M})=\U(n,n).\]
The usual symplectic group
$\mathrm{Sp}_{2n}(\R)$ is precisely the subgroup of 
$\U(n,n)$
of all those matrices with real entries. We will write $\SU(n,n)$ for the subgroup 
of $U(n,n)$ of matrices with determinant $1$.


\subsection{The representation}
\label{subsec:rep}
We are now ready to revisit Looijenga's arguments from \cite[Section~4]{Lo} to study the set of matrices 
$M_q(\varphi)$ for $\varphi\in\Ig$ (compare Section~\ref{sec:finite-covers}). We need the
following set up.

Fix an integer $q\geq 3$ and let $G$ be a cyclic group of order $q$.
Denote by $k = \mathbb{Q}[\zeta_q]$ the extension of $\Q$ by the $q$-th roots of unity. 
We identify $k$ with a subfield of the complex numbers. 
There is a map
\[ \iota: \Z[G] \to k \subset \C \] mapping a generator of $G$ to a
primitive $q$-th root of unity\footnote{in the notation of \cite{Lo},
  we have $k \iso K_G$, the field defined by the trivial cyclic
  quotient of $G$}.  We often consider the image of $\iota$ as
contained in the complex numbers rather than in the abstract field
$k$. Note that $\iota$ is not unique. Different choices of $\iota$ differ by an
element of the Galois group of the field $k$. If $q$ is prime then
$\iota$ can be chosen to send any given generator $1\in G$ to any
prescribed 
$q$--th root of unity. This freedom of choice will be important later
(in Section 6). Until then, all arguments will work for any choice of $\iota$. 

Under $\iota$, 
the involution on $\Z[G]$ induced by $\overline{g}=g^{-1}$ corresponds to complex
conjugation: 
\[ \iota(\overline{p}) = \overline{\iota(p)}, \quad\forall p \in \Z[G].\]
If $x \in \Z[G]^n$, 
we denote by $\iota x \in \C^n$ the
vector obtained by applying $\iota$ coordinate-wise.

As explained in Section~\ref{sec:matrix}, we may assume without loss of 
generality that $S_q\to S$ is the regular cyclic
cover of degree $q$ with deck group $G$ induced by algebraic
intersection number mod $q$ with the curve $\alpha_g$. We use the notation from 
that section.

The homology
$H_1(S_q;\Z)$ is equipped with the usual intersection form $(\cdot,
\cdot)$ as well as with a skew-Hermitian form $\Phi : H_1(S_q;\Z)
\times H_1(S_q;\Z) \to \Z[G]$ (sometimes called Reidemeister pairing)
given by
\[ \Phi(x,y) = \sum_{g\in G} (x,gy)g. \]

Because $G$ preserves the submodule 
$F_q \subset H_1(S_q;\Z)$, the bilinear form $\Phi$ 
defines a skew-Hermitian form on $F_q \iso \Z[G]^{2g-2}$.

Furthermore, we have for all $x,y\in \Z[G]^{2g-2}$ and $\varphi\in\Ig$
\begin{equation}\label{invariance}
\Phi( M_q(\varphi)x, M_q(\varphi)y ) = 
\Phi( x,y ) \end{equation}
since each lift of $\varphi$ preserves the intersection form $(\cdot,\cdot)$
and commutes with the deck group action.

Recall from Subsection~\ref{sec:finite-covers} the definition of the 
classes $\hat a_i,\hat b_k\in H_1(S_q;\mathbb{Z})$.
In what follows, we choose a basis $\{x_i,y_i\}_{i=1}^{g-1}$ for $\C^{2g-2}$ such that: 
\[\iota(\hat{a_i})=x_i \text{ and } \iota(\hat{b_i})=y_i \]
for all $i=1,\ldots, g-1$. We also define a skew-Hermitian form $\langle\cdot,\cdot\rangle$ on $\C^{2g-2}$ by
\[\langle x_i,y_i\rangle = \delta_{ij} \text{ and } \langle x_i,x_j\rangle =\langle y_i,y_j\rangle=0 \]
for all $i,j=1,\ldots, g-1$.

We now have the following
\begin{lemma}\label{lem:symplectic-to-symplectic} For all $x,y\in H_1(S_q;\Z)$:
  \[\iota \Phi(x,y) = \langle \iota x, \iota y\rangle \]
\end{lemma}
\begin{proof} We have:
\[\Phi(\hat{a}_i,\hat{b}_j) =  \delta_{ij} \text{ and } (\hat{a}_i,\hat{a}_j)=(\hat{b}_i,\hat{b}_j)=0  \]
The form $\Phi$ is skew-Hermitian and $\iota$ is linear by definition, hence the lemma.
\end{proof}

By Lemma~\ref{lem:symplectic-to-symplectic} and equation (\ref{invariance}), 
for all $\varphi\in\Ig$
the matrix $\iota M_q(\varphi)$ is contained in the group $\U(g-1,g-1)$
of skew Hermitian automorphisms of $\C^{2g-2}$. 

The assignment $\varphi\mapsto \iota M_q(\varphi)$ is 
not a representation of $\Ig$ since the matrix $M_q(\varphi)$ depends on
a choice of a lift of $\varphi$ to our cyclic cover, and therefore is defined only up to a unit in $\Z[G]$ (compare
Section~\ref{sec:matrix}). The (complex) matrix $\iota M_q(\varphi)$ is thus also only defined up to multiplication by a $q$--th root of unity.

To avoid these problems, let
$\widehat{\mathcal{I}}_g^{(q)}$ be the group of all lifts to $S_q$ of all
elements of $\Ig$.  We then have a short exact sequence
\[ 1 \to G \to \widehat{\mathcal{I}}_g^{(q)} \to \Ig \to 1 \]
and an actual representation
\[\rho_q:\widehat{\mathcal{I}}_g^{(q)}\to \U(g-1,g-1). \]

In \cite{Lo}, Looijenga studied this representation. More generally,  
denote by ${\rm Mod}(S)^{(q)}$ the 
finite index subgroup of 
${\rm Mod}(S)$ of all elements which admit a lift to $S_q$. The group
$\Ig$ is a normal subgroup of ${\rm Mod}(S)^{(q)}$.
As before, there is an exact sequence
\[1\to G\to \widehat{{\rm Mod}}(S)^{(q)}\to {\rm Mod}(S)^{(q)}\to 1.\]
Using the action on homology, we obtain a representation
$\hat\rho_q:\widehat{{\rm Mod}}(S)^{(q)}\to \U(g-1,g-1)$ extending
the representation $\rho_q$. 

Denote by $\U^\sharp(g-1,g-1)$ the subgroup of $\U(g-1,g-1)$ consisting
of matrices whose determinant is equal to a square of a $q$-th root of unity.
Note that $\SU(g-1,g-1)$ is a finite index subgroup of $\U^\sharp(g-1,g-1)$.
Theorem 2.4 of \cite{Lo} states the following.

\begin{theorem}\label{arithmeticity}
Let $g\geq 3,q\geq 3$ and let $R_q$ be the ring of integers in the number field 
$\mathbb{Q}[\zeta_q]$. Then 
\[\hat \rho_q\left(\widehat{\rm Mod}(S)\right)=U^\sharp(g-1,g-1;R_q).\]
\end{theorem}

Here the group $U^\sharp(g-1,g-1;R_q)$ is just the subgroup of 
$U^\sharp(g-1,g-1)$ of matrices with coefficients in $R_q$.

\subsection{Denseness}

We are now ready to prove that the image of the representation 
$\rho_q$ defined above is Zariski dense in  
$\SU(g-1,g-1)$.

The proof of the following proposition was suggested to us by an anonymous
referee and replaces an earlier argument which followed Looijenga's paper \cite{Lo}.

\begin{proposition}\label{prop:torelli-image-dense}
  Let $g\geq 3$ and $q\geq 3$; then 
  \[ \rho_q\left(\widehat{\mathcal{I}}_g^{(q)}\right) \cap \SU(g-1,g-1) \]
  is a finite index subgroup of $\SU(g-1,g-1;R_q)$. 
\end{proposition}
\begin{remark}
  The proposition is also true for $g=2$ and $q=5$ or $q\geq 7$ (compare \cite{Lo} and the proof below).
Since we only need the $g \geq 3$ case, we do not give details.
\end{remark}
\begin{proof} By Theorem \ref{arithmeticity}, the image of 
the group $\widehat{\rm Mod}(S)^{(q)}$ under the representation $\hat \rho_q$
contains the group $\SU(g-1,g-1;R_q)$. Now for $g\geq 3$, the group
$\SU(g-1,g-1;R_q)$ is arithmetic and of $\mathbb{Q}$-rank at least two. 
Furthermore, it contains the group 
\[ A=\rho_q\left(\widehat{\Ig}^{(q)}\right)\cap \SU(g-1,g-1;R_q) \] as a normal subgroup.
Thus by Margulis's normal subgroup theorem, either $A$ is finite,
or it has finite index. Thus we have to show that the group $A$ is infinite.

Write ${\mathcal M}=\hat \rho_q^{-1}(\SU(g-1,g-1;R_q))$; 
since $U^\sharp(g-1,g-1;R_q)$ is a finite central extension of 
$\SU(g-1,g-1;R_q)$, ${\mathcal M}$ 
is a finite index normal subgroup 
of $\widehat{\rm Mod}(S)^{(q)}$ which intersect the Torelli group $\hat{\Ig}^{(q)}$ in a finite index
subgroup $\mathcal{P}$.

Let us denote by $\Pi:{\mathcal M}\to \Sp(2g,\mathbb{Z})$ the natural projection. As
${\mathcal M}$ is a subgroup of ${\rm Mod}(S)$ of finite index, 
$\Pi({\mathcal M})={\mathcal H}$ has finite index in $\Sp(2g,\mathbb{Z})$ and hence is a lattice
of the simple Lie group $\Sp(2g,\mathbb{R})$ of rank $g\geq 3$.
Taking the quotient of ${\mathcal M}$ by the kernel ${\mathcal K}<{\mathcal M}$ of the 
homomorphism $\hat \rho$ 
gives rise to an exact sequence
\begin{equation}\label{exact}
1\to {\mathcal P}/{\mathcal K}\cap {\mathcal P} \to {\mathcal M}/{\mathcal K}\to 
{\mathcal H}/\Pi({\mathcal K})\to 1
\end{equation}
where the group ${\mathcal M}/{\mathcal K}$ is isomorphic to $\SU(g-1,g-1;R_q)$. 

Now if the image of the group $\widehat{\Ig}^{(q)}$ under the representation $\rho_q$ is finite then 
${\mathcal P}/{\mathcal K}$ is a finite group. Thus
this sequence describes the arithmetic group $\SU(g-1,g-1;R_q)$ as an extension of 
the group ${\mathcal H}/\Pi({\mathcal K})$ by a finite group. This implies that 
${\mathcal H}/\Pi({\mathcal K})$ is infinite. Margulis's normal subgroup
theorem now shows that $\Pi({\mathcal K})$ is a finite normal subgroup of the higher
rank lattice ${\mathcal H}$.

Now ${\mathcal H}<\Sp(2g,\mathbb{R})$ is Zariski dense and therefore a finite normal subgroup of 
${\mathcal H}$ is normalized by the entire simple group $\Sp(2g,\mathbb{R})$. 
Then the group has to be central in $\Sp(2g,\mathbb{R})$ and hence either it is trivial,
or it equals $\pm 1$. As a consequence, the group 
${\mathcal H}/\Pi({\mathcal K})$ either equals a lattice
in $\Sp(2g,\mathbb{R})$, or a lattice in the quotient 
$\mathrm{P}\Sp(2g,\mathbb{R})$ 
of this group by its center.

To summarize, under the assumption that $\rho_g({\mathcal P})$ is not a
subgroup of  
$\SU(g-1,g-1;R_q)$ of finite index we conclude that the arithmetic group 
$\SU(g-1,g-1;R_q)$ is a finite extension of a lattice in $\mathrm{P}\Sp(2g,\mathbb{R})$. 
But $\SU(g-1,g-1;R_q)$ is an irreducible lattice in a higher rank
semi-simple Lie group $G$ which does not contain any factor 
locally isomorphic to the higher rank simple Lie group
$\Sp(2g,\mathbb{R})$. 
This contradicts Margulis' super-rigidity theorem for lattices.
\end{proof}

\begin{remark} 
Let $\epsilon:\mathbb{Q}[\zeta_q]\to \mathbb{Z}$ be the $\mathbb{Z}$-linear 
augmentation map defined by 
$\epsilon(\zeta_q)=1$ for ever $q$-th root of unity $\zeta_q$. 
It follows from the proof of Proposition \ref{prop:torelli-image-dense}
that the image of $\Ig$ under $\rho_q$ intersects 
 $\SU(g-1,g-1;R_p)$ in the finite index
normal subgroup defined as the preimage of the identity
under the map which associates to a matrix its coordinate-wise image under 
$\epsilon$. We do not know whether the image
coincides with this group. 
\end{remark}

In the following corollary, we view $\SU(g-1,g-1;R_q)$ as a 
countable subgroup of the
Lie group $\SU(g-1,g-1)$ (and not as an arithmetic group).

\begin{corollary}\label{zarsiki}
For $g\geq 3, q\geq 3$ 
the group $\rho_q\left(\widehat{\mathcal{I}}_g^{(q)}\right) \cap \SU(g-1,g-1) $
is Zariski dense in $\SU(g-1,g-1)$.
\end{corollary}
\begin{proof} By Proposition \ref{prop:torelli-image-dense}, 
if the degree over $\mathbb{Q}$ of the number field
$\mathbb{Q}[\zeta_q]$ is at least four then 
$\rho_q\left(\widehat{\mathcal{I}}_g^{(q)}\right) \cap \SU(g-1,g-1) $
is a lattice in an algebraic group which has $\SU(g-1,g-1)$ as a factor.
In this case the image group 
is dense in the usual topology. 

Otherwise, as $q\geq 3$, 
the degree over $\mathbb{Q}$ of $\mathbb{Q}[\zeta_q]$ is two and
$\rho_q\left(\widehat{\mathcal{I}}_g^{(q)}\right) \cap \SU(g-1,g-1) $ is a lattice
in $\SU(g-1,g-1)$ and hence Zariski dense as well.
\end{proof}

\subsection{The action on subspaces: Part I}\label{sec:lagrangians1}
The core ingredient in the proof of the main theorem will be to
control the determinant of $B_q(\varphi)$ (introduced in
Section~\ref{sec:matrix}).  This will be done using random walks on
algebraic groups in Section~\ref{sec:conditions}. In this section we
describe how to detect this determinant, using the action of
$\U(g-1,g-1)$ on subspaces of $\C^{2g-2}$.


Recall that the exterior product of a basis of a $(g-1)$--dimensional subspace
of $\mathbb{C}^{2g-2}$ is a pure 
vector in $\wedge^{g-1}\mathbb{C}^{2g-2}$. Thus the action of 
$\U(g-1,g-1)$ on the Grassmannian of half-dimensional subspaces of 
$\mathbb{C}^{2g-2}$ is encoded in the natural representation of 
$\U(g-1,g-1)$ on $\wedge^{g-1}\mathbb{C}^{2g-2}$. 
This representation is irreducible (see Section \ref{sec:lagrangians2}).

Let $\{x_i,y_i\}_{i=1}^{g-1}$ denote a basis of $\C^{2g-2}$ with respect to which $\langle \cdot,\cdot\rangle$ takes the standard form 
(see Subsection \ref{subsec:rep}). Let $e,f\in \wedge^{g-1}\mathbb{C}^{2g-2}$ be defined by
\[e = x_1\wedge\cdots\wedge x_{g-1} \;\;\text{and}\;\; f=y_1\wedge \cdots \wedge y_{g-1}.\]

Recall that we want to control the lower left block in the image of our random element of $\widehat{\mathcal{I}}_g^{(q)}$ under $\rho_q$. The first observation is that this determinant appears as the $f$-coefficient of $\rho_q(\hat\varphi)e$. To this end, let $(\cdot,\cdot)$ denote the Hermitian 
inner product on $\wedge^{g-1}\C^{2g-2}$ corresponding to the Hermitian inner product on $\mathbb{C}^{2g-2}$ for which the 
basis $\{x_i,y_i\}_{i=1}^{g-1}$ is orthonormal.

\begin{lemma}\label{lem:determinant} Let $\varphi \in \Ig$ be arbitrary, and let 
$\hat\varphi \in \widehat{\mathcal{I}}_g^{(q)}$ be any lift of $\varphi$. Then
\[|(\rho_q(\hat\varphi)e,f)| = |\iota\det(B_q(\varphi))|.\]
\end{lemma}

\begin{proof} 
First note that a lift $\hat\varphi$ is defined up to multiplication by a deck group element,
which under $\iota$ maps to multiplication by a complex number of absolute value $1$. Thus
both sides of the inequality are independent of the choice of lift. Similarly, we can choose our
preferred lift of $\varphi$ to compute $\iota B_q(\varphi)$.

Let
\[\mathcal{E}= \left\{ x_{i_1}\wedge\cdots\wedge x_{i_k} \wedge y_{j_1}\wedge\cdots\wedge y_{j_{g-1-k}}\mid \substack{\displaystyle{0\leq k\leq g-1,\; i_{1}<\ldots <i_k,}\\ \displaystyle{j_1<\ldots<j_{g-1-k}}}\right\}.\]
$\mathcal{E}$ is a orthonormal basis for $\wedge^{g-1}\C^{2g-2}$. Because $f$ is orthogonal to all elements of $\mathcal{E}$ not equal to it, we obtain
\begin{eqnarray*}
(\rho_q(\varphi)e,f) &=&  \sum_{i_1,\ldots,i_{g-1}=1}^{g-1} \left(\left(B_q(\varphi)\right)_{1,i_1}y_{i_1}\wedge\ldots\wedge \left(B_q(\varphi)\right)_{g-1,i_{g-1}}y_{i_{g-1}}, f \right) \\
&=& \sum_{i_1,\ldots,i_{g-1}=1}^{g-1} (y_{i_1}\wedge \ldots \wedge y_{i_{g-1}}, f)\prod_{j=1}^{g-1}\left(B_q(\varphi)\right)_{j,i_j} .\\
\end{eqnarray*}
The inner products in the terms satisfy
\[(y_{i_1}\wedge \ldots \wedge y_{i_{g-1}}, f) = \left\{
\begin{array}{ll} 
\varepsilon(i_1,\ldots,i_{g-1}) & \text{if the map}\; j\mapsto i_j\;\text{is bijective} \\
0 & \text{otherwise}
\end{array} \right. \]
where $\varepsilon(i_1,\ldots,i_{g-1})$ denotes the sign of the map $j\mapsto i_j$ viewed as a permutation. As such, we can view the sum above as a sum over elements of the symmetric group on $g-1$ letters $\mathfrak{S}_{g-1}$. We obtain:
\begin{eqnarray*}
(\rho_q(\varphi)e,f) &=& \sum_{\pi\in\mathfrak{S}_{g-1}} \varepsilon(\pi) \prod_{j=1}^{g-1}\left(B_q(\varphi)\right)_{j,\pi(j)},
\end{eqnarray*}
where $\varepsilon(\pi)$ denotes the sign of a permutation $\pi\in\mathfrak{S}_{g-1}$. The expression above is the Leibniz formula for the determinant of $B_q(\varphi)$.
\end{proof}

\section{Mahler measures}
\label{sec:algebra}
In this section we review some basic properties of the 
logarithmic Mahler measure of polynomials and
establish conditions which are
satisfied in the case that the polynomials $\det B_\infty(\varphi)$ from Section~\ref{sec:matrix} have 
logarithmic Mahler
measure zero.

\subsection{Mahler measure obstructions}
First, we collect some classical facts on integral polynomials in a single variable and their Mahler measures.
 
Let $P$ be a polynomial in one variable with integral coefficients. 
The \emph{(logarithmic) Mahler measure} of $P$ is defined as
\[ m(P) = \frac{1}{2\pi}\int_0^{2\pi} \ln(|P(\exp(i\theta))|)d\theta .\]
Note that $m(t^kP) = m(P)$ for all $k\in \Z$.

We let $M(P) = \exp(m(P))$ be the multiplicative Mahler measure of $P$. If we speak simply of
Mahler measure, we will usually mean the logarithmic Mahler measure.
If we
write $P(t) = a\prod_i(t-\alpha_i)$ with $\alpha_i\in \mathbb{C}$ 
then by Jensen's formula, we have
\[ M(P) = |a|\prod\max(1, |\alpha_i|).\]

For $n\in\mathbb{N}$, the $n^{th}$ \emph{cyclotomic
  polynomial} $\Phi_n$ in $t$ is the unique irreducible polynomial (over $\mathbb{Z}$) that is a
divisor of $t^{n}-1$ but not of $t^k-1$ for any $k<n$, $k\in \mathbb{N}$. As such, its roots are exactly the $n^{th}$ primitive roots of
unity. This in turn implies that
\[\deg(\Phi_n) = \varphi(n)\]
for all $n\in\mathbb{N}$, where $\varphi:\mathbb{N}\to\mathbb{N}$ denotes Euler's totient function.

The following proposition is a well-known application of a result of Kronecker \cite{Kr} (compare e.g. \cite[Theorem~1.31]{EW}).
\begin{proposition}\label{prop:mahler-measure-one}
  Let $P(t)$ be a polynomial in one variable with integral coefficients.
  Then $m(P(t)) \geq 0$, with equality if and only if $P(t)$ is a product
  of cyclotomic polynomials and powers of $t$.
\end{proposition}

From
the proposition above
we will derive a condition on polynomials of Mahler measure zero. Before we can prove it, 
we need two classical
results on cyclotomic polynomials and Euler's totient function. To
this end, let $\mathrm{c}(\Phi_m)$ denote the maximum of the absolute
values of the coefficients of $\Phi_m$. The following bound on
$\mathrm{c}(\Phi_m)$ is due to Maier.

\begin{proposition}\label{prop:cyclotomic-coefficients}\cite{Ma}
Let $\eta:\mathbb{N}\to\mathbb{R}$ be a function so that
\[\eta(n) \to \infty \]
as $n\to\infty$. Then for all but finitely many $m\in\mathbb{N}$ we have
\[\mathrm{c}(\Phi_m) \leq m^{\eta(m)}.\]
\end{proposition}

Of course, the set of $m\in\mathbb{N}$ for which the above bound above does not hold 
depends on the choice of function $\eta$. 

Finally, we need the following bound for the totient
function (see e.g. \cite[Theorem 327]{HW}). In fact, sharper estimates on $\varphi(n)$ are available, but the following is enough for our purpose.
\begin{proposition}\label{prop:totient-function} For every $\delta\in (0,1)$ we have:
\[\frac{\varphi(n)}{n^\delta} \to \infty\]
as $n\to\infty$.
\end{proposition}

We are now ready to prove our condition on polynomials.
\begin{proposition}\label{prop:mahler-measure-constraint} For every $\alpha>0$ there exists a finite set 
$K_\alpha\subset\mathbb{N}$ with the following property. If $P\in\Z[t]$ satisfies $m(P)=0$ then
\begin{itemize}
\item either $\Phi_k$ is a factor of $P$ for some $k\in K_\alpha$,
\item or for all $\xi\in\C$ with $\vert\xi\vert=1$ we have
\[ \vert P(\xi)\vert \leq \exp(\alpha \deg(P)). \]
\end{itemize}
\end{proposition}

\begin{proof} 
  Proposition \ref{prop:totient-function} (applied to $\delta=1/2$)
  allows us to find some $n_0\in \mathbb{N}$ so that
\[ \varphi(n) > \sqrt{n},\]
for all $n\geq n_0$. 
 Define $\eta:\mathbb{N}\to\mathbb{R}$ by
\[\eta(n) = \alpha\cdot \frac{\sqrt{n}}{\log(n)}-1.\]
Using Proposition \ref{prop:cyclotomic-coefficients}, we obtain a finite set $K'\subset\mathbb{N}$ 
so that the coefficients of all $\Phi_m$ for $m\notin K'$ are bounded in absolute value 
by $m^{\eta(m)}$. This implies that for $m\notin K'$ we have
\[\vert\Phi_m(\xi)\vert \leq \varphi(m)\cdot m^{\eta(m)} \]
for all $\xi\in\C$ with $\vert \xi\vert=1$. The same holds trivially for the polynomial $P(t)=t$.

Define $K_\alpha\subset \mathbb{N}$ by
\[K_\alpha = K'\cup \{0,1,\ldots, n_0\}.\]
We claim that the proposition holds for this set. As such, we need to check the inequality 
in the statement of the proposition for all polynomials 
of Mahler measure zero not containing any factor $\Phi_k$ for some $k\in K_\alpha$. 

Let $P\in\Z[t]$ a polynomial with $m(P)=0$. Proposition \ref{prop:mahler-measure-one} tells us that we may write
\[P(t) = t^k\prod_{i=1}^r\Phi_{m_i}(t) \]
for some $m\in\mathbb{N}^r$ and $k,r\in\mathbb{N}$. Assume that $m_i\notin K_\alpha$ for all $i=1,\ldots,r$. We have
\[\deg(P) = k + \sum_{i=1}^r \varphi(m_i) >  \sum_{i=1}^{r} \sqrt{m_i}.\]
By the choice of $K_\alpha$, 
for all $\xi\in\C$ with $\vert \xi\vert =1$ we obtain that
\[\vert P(\xi)\vert \leq \prod_{i=1}^{r}\varphi(m_i)m_i^{\alpha \sqrt{m_i}/\log(m_i)-1}\leq 
\prod_{i=1}^r m_i^{\alpha\cdot \sqrt{m_i} / \log(m_i)}, \]
where we have used that $\varphi(n)\leq n$ for all $n\in\mathbb{N}$. Hence 
\[\vert P(\xi)\vert \leq \exp\left( \alpha \sum_{i=1}^r \sqrt{m_i}  \right)\leq \exp(\alpha\cdot \deg(P)).\]
\end{proof}


\section{Finishing the proof}
\label{sec:finishing}

In this section we finish the proof of our main theorems.  We will
again restrict to the case of random walks on the Torelli group. The case of
the homology stabilisers is completely analogous.

Let $S$ be a surface of genus $g\geq 3$ and let $\sigma:\pi_1(S) \to
\mathbb{Z}$ be a surjection as in Section~\ref{sec:setup} so that
${\rm ker}(\sigma)\supset K={\rm ker}(\pi_1(S)\to \pi_1(V))$.
Let $S_\infty, S_q$ be the covers considered in Section~\ref{sec:matrix} induced by $\sigma$.
For $\varphi \in \Ig$ and $q\in \mathbb{N}$ 
let $(N_\varphi)_q$ denote the cover of
$N_\varphi$ defined by the map $\pi_1(N_\varphi) \to \mathbb{Z}/q\mathbb{Z}$
which factors through $\sigma$. Say that \emph{$\varphi$ has property  $\mathrm{E}_\sigma$} if
  \[ \lim_{q\to \infty} \frac{\log | H_1((N_\varphi)_q;\mathbb{Z})_{\mathrm{tor}} |}{ q } \]
  exists and is positive.

Our goal is to prove:
\begin{theorem}\label{thm:main-theorem-plain}
Let $\mu$ be any finitely supported measure on the
  Torelli group $\Ig<{\rm Mod}(S)$ whose support generates
  $\Ig$ as a semigroup. Then we have
\[ \mu^{\ast n}\left(\left\{\varphi \in \Ig,  \,\varphi \text{ has property }\mathrm{E}_\sigma\right\} \right) \to 1 \]
as $n\to \infty$.


\end{theorem}

By Proposition~\ref{prop:presentation-matrix}, $B_\infty(\varphi)$ is a presentation matrix
for $H_1((N_\varphi)_\infty;\mathbb{Z})$. We can now apply \cite[Theorem~3.1]{R}
(where $m=1$, $\varphi=\sigma$, $\hat{X}=(N_\varphi)_\infty, X_q = (N_\varphi)_q$)
and obtain the following.
\begin{proposition}\label{prop:criterion-exponential-growth}
  With notation as above, if $m(\det B_\infty(\varphi)) > 0$ 
  then 
  \[ \lim_{q\to \infty} \frac{\log | H_1(X_q;\mathbb{Z})_{\mathrm{tor}} | }{ q } = m(\det B_\infty(\varphi))\]
  exists and is positive.
\end{proposition}

Thus, if $\varphi$ does not have $E_\sigma$, then we have that $m(\det B_\infty(\varphi))=0$. 
The rest of this section will be devoted to showing that with probability converging to one,
the logarithmic Mahler measure of $\det B_\infty(\phi))$ does not vanish.  
In order to be able to apply results on random walks on Lie groups later on, we first translate this into a condition on the image of (lifts of) $\varphi$ under the representations $\rho_m$ for some $m\in\mathbb{N}$.

\subsection{Matrix Conditions}\label{sec:conditions}
Recall from Lemma~\ref{lem:determinant} that the determinant of $B_q(\varphi)$ can be computed from the action of $\rho_q(\hat\varphi)$ on $\wedge^{g-1}\C^{2g-2}$.
Using the above, the condition in Proposition \ref{prop:mahler-measure-constraint} translates as follows. Here, $e$ and $f$ are as in Section~\ref{sec:lagrangians1}.
\begin{proposition}\label{prop:matrix-condition} 
  Let $\alpha>0$ and let $\mu$ be a finitely supported
  probability measure on $\Ig$ whose support generates $\Ig$ as a semigroup. There exists a finite set
  $K_{\alpha,\mu}\subset \mathbb{N}$ such that the following holds.

  Let $n$ be arbitrary and
  $\varphi\in\mathrm{supp}\left(\mu^{*n}\right)\subset\Ig$ with
  $m(\det B_\infty(\varphi))=0$. Then
\begin{itemize}
\item either $\Phi_k$ is a factor of $\det B_\infty(\varphi)$ for some $k\in K_{\alpha,\mu}$, and
  therefore $(\rho_k(\hat\varphi)e,f) = 0$ for that $k$ and any lift $\hat\varphi$ of $\varphi$ to $S_k$,
\item or for all $q\in\mathbb{N}$ we have
\[\ \vert (\rho_q(\hat\varphi)e,f)\vert \leq \exp(\alpha \cdot n) \]
for any lift $\hat\varphi$ of $\varphi$ to $S_q$.
\end{itemize}
\end{proposition}

\begin{proof} 
Let $S=\mathrm{supp}(\mu)\subset\Ig$ and let us choose lifts of all $s\in S$ so that the corresponding $M_\infty(s)$ are matrices whose coefficients are polynomials (as opposed to Laurent polynomials). Because $S$ is finite, this may be achieved by multiplying by some large power of $t$. 

We will first derive a degree bound on $\det(B_\infty(\varphi))$ for $\varphi\in\mathrm{supp}\left(\mu^{*n}\right)$. Let 
\[d_\mu = \max\left\{\deg\left(\left(M_\infty(s)\right)_{ij}\right)\vert\; i,j=1,\ldots,2g-2,\;s\in S\right\}.\]
Any $M_\infty(\varphi)$ for $\varphi\in\mathrm{supp}\left(\mu^{*n}\right)$ is obtained by multiplying at most $n$ elements of $\{M_\infty(s)\vert\; s\in S\}$. Hence the coefficients of $M_\infty(\varphi)$ have degree at most $d_\mu\cdot n$. The polynomial $\det(B_\infty(\varphi))$ is obtained by taking the determinant of a $(g-1)\times (g-1)$ block of $M_\infty(\varphi)$. As such, we obtain
\[\deg(\det(B_\infty(\varphi))) \leq (g-1)\cdot d_\mu \cdot n. \]

Set $\alpha'=\alpha/((g-1)d_\mu)$, and let $K_{\alpha'}\subset\mathbb{N}$ be the finite set
obtained from Proposition~\ref{prop:mahler-measure-constraint} applied with $\alpha'$. We
put $K_{\alpha, \mu} = K_{\alpha'}$.

\smallskip
Now suppose that $\varphi$ is so that $m(\det B_\infty(\varphi))=0$. Then by Proposition~\ref{prop:mahler-measure-constraint}, 
either $\Phi_k$ divides $\det(B_\infty(\varphi))$ for some $k\in K_{\alpha'}$ or
\[\vert \det(B_\infty(\varphi))(\xi)\vert \leq \exp(\alpha' \deg(\det(B_\infty(\varphi))))\]
for all $\xi$ with $\vert\xi\vert=1$.
In the first case, every primitive $k$--th root of unity is a zero of
$\Phi_k$ and hence of $\det(B_\infty(\varphi))$. Since $\iota\det(B_k(\varphi))$ is obtained from
$\det(B_\infty(\varphi))$ by evaluating at a primitive $k$--th root of unity, this implies the
claim (using Lemma~\ref{lem:determinant}).

In the latter case, the same argument yields that, for $\zeta_q$ a primitive $q$-th root of unity,
\[ \vert (\rho_q(\hat{\varphi})e, f) \vert = \vert \iota\det(B_q(\varphi))
\vert = \vert\det(B_\infty(\varphi))(\zeta_q)\vert\]\[ \leq \exp(\alpha' \deg(\det(B_\infty(\varphi))))
\leq \exp(\alpha n).\] 
\end{proof}

\subsection{Set-up and walks on $\widehat{\mathcal{I}}_g^{(q)}$}
\label{sec:lifting-walks}

Let $q\geq 3$, and $\zeta_q$ be a primitive $q$-th root of unity. Then 
$k=\mathbb{Q}[\zeta_q]$
is the cyclotomic field of degree $\varphi(q)$ over $\mathbb{Q}$. 
Recall from Subsection~\ref{subsec:rep} that we have a ring 
morphism 
\[\iota:\mathbb{Z}[G]\to k\subset \mathbb{C}\]
where $G$ is the cyclic group with $q$ elements.
It image is contained in the ring of integers of $k$. The morphism $\iota$ 
 depends on the choice of 
$\zeta_q$. 

By the results of Section~\ref{sec:c-side}, the ring morphism $\iota$ induces a representation
\[\rho_q:\widehat{\mathcal{I}}_g^{(q)} \to \U(g-1,g-1) \]
whose image is contained in the subgroup $U^\sharp (g-1,g-1)$ of all elements 
with determinant a square of a $q$-th root of unity. Here as before, 
$\widehat{\mathcal{I}}_g^{(q)}$ is the group of all lifts of elements in the Torelli group to
the surface $S_q$, and this group fits into the exact sequence
\[ 1 \to G \to \widehat{\mathcal{I}}_g^{(q)} \to \Ig \to 1\]

In the statement of our main theorem, we use a random walk on
the Torelli group $\Ig$. To apply the representations $\rho_q$ and the
criteria for Mahler measure $0$ above, we have to work
with random walks on a finite number of groups
$\widehat{\mathcal{I}}_g^{(q)}$. The rest of this section is devoted to
explaining how to pass from one to the other.

Let $\mu$ be a probability 
measure on $\Ig$ whose support is finite and generates $\Ig$ as a semigroup.
Define a measure $\zeta$ on 
$\widehat{\mathcal{I}}_g^{(q)}$ via
\[ \zeta(h) = |G|^{-1}\mu(hG). \]
Then $\zeta$ is a measure whose finite support generates $\widehat{\mathcal{I}}_g^{(q)}$ as a semigroup.


\begin{lemma}\label{lem:average}
Let $\pi:\widehat{\mathcal{I}}_g^{(q)}\to\Ig $ be the quotient map.
Then $\zeta^{*n}(\pi^{-1}A)=\mu^{*n}(A)$ for any $A\subset \Ig$.
\end{lemma}
\begin{proof}
Let $Z\subset \widehat{\mathcal{I}}_g^{(q)}$ be a complete set of coset
representatives for $\Ig$ (under the quotient map $\pi$).

Then
\begin{eqnarray*}
  \zeta^{*n}(\pi^{-1}A) &=&
  \sum_{h_{1},\ldots,h_{n}\in \widehat{\mathcal{I}}_g^{(q)},h_{1}\cdots h_{n}\in
    \pi^{-1}A}\zeta(h_{1})\cdots\zeta(h_{n}) \\
  &=&\sum_{a_{1},\ldots ,a_{n}\in Z,g_{1},\ldots,g_{n}\in G, a_{1}\cdots a_{n}G\in
    A}\zeta(a_{1}g_{1})\cdots\zeta(a_{n}g_{n}) \\
  &=& \sum_{a_{1},\ldots,a_{n}\in Z,g_{1},\ldots,g_{n}\in G, g_{1},\ldots,g_{n}G\in
    A}\frac{\mu(a_{1}g_1G)}{|G|}\cdots\frac{\mu(a_{n}g_nG)}{|G|} \\
  &=&\sum_{a_{1},\ldots,a_{n}\in Z,a_{1}\cdots a_{n}G\in A}
\mu(a_{1}G)\cdots \mu(a_{n}G) \\
  &=&\sum_{l_{1},\ldots,l_{k}\in
\Ig,l_{1}\cdots l_{n}\in A}\mu(l_{1})\cdots\mu(l_{n}) \\
 &=& \mu^{*n}(A).
\end{eqnarray*}
\end{proof}

\subsection{The action on subspaces: Part II}\label{sec:lagrangians2}

Before we can complete the proof, we need to collect some final facts on the
representation $\rho_q:\widehat{\mathcal{I}}_g^{(q)}\to U(g-1,g-1)$ and its action on subspaces.
Recall that we are interested in the representation 
\[\sigma_q:\widehat{\mathcal{I}}_g^{(q)}\to \GL(\wedge^{g-1}\C^{2g-2}).\] 

Recall that a representation $\rho$ of a group $G$ on a finite dimensional complex 
vector space $V$ is \emph{strongly irreducible} if it is irreducible and if 
$\rho(G)$ does not preserve any finite union of proper subspaces. 

The following proposition is standard.
\begin{proposition}\label{prop:irreducibility}
  The representation of $\SU(g-1,g-1)$ on 
  $\wedge^{g-1}\C^{2g-2}$ induced by the standard action on $\C^{2g-2}$
is strongly irreducible. 
 The image of $\SU(g-1,g-1)$ in
  $\mathrm{GL}(\wedge^{g-1}\C^{2g-2})$ is semi-simple and Zariski closed. 
\end{proposition}
\begin{proof}
  It is a classical fact that for any $k$, the representation of
  $\mathrm{SL}(d,\C)$ on $\wedge^k\C^d$ is irreducible (compare
  e.g. \cite[\S 15.2]{FH}). Now, $\SU(g-1,g-1)$ is a simple
  non-compact real Lie group with Lie algebra
  $\mathfrak{su}_{g-1,g-1}$. We have that
  $\mathfrak{su}_{g-1,g-1} \otimes \mathbb{C} = \mathfrak{sl}_{2g-2}$.
  This implies that restrictions of irreducible representations of
  $\mathrm{SL}(2g-2,\C)$
  to $\SU(g-1,g-1)$
  are irreducible (compare the discussion in \cite[\S 26.1,
  p. 439]{FH}).  Since $\SU(g-1,g-1)$
  is connected, this representation is in fact strongly irreducible.

Now the map
  $\SU(g-1,g-1) \to \mathrm{GL}(\wedge^{g-1}\C^{2g-2})$ is
  algebraic, and images of algebraic groups under algebraic maps are
  Zariski closed. Furthermore, as $\SU(g-1,g-1)$ is simple, the same holds true for the 
  image group.   
\end{proof}

Let $\Gamma$ be any countable group. 
A representation $\sigma:\Gamma\to \GL(V)$ on a complex vector space $V$ 
is called \emph{proximal}
if there exists a sequence $(g_n)\subset \sigma(\Gamma)\subset \GL(V)$ 
and a sequence
$(\lambda_n)\subset \mathbb{C}$ such that
\[\pi=\lim_{n}\lambda_ng_n\]
is an endomorphism of rank one. 
 
\begin{lemma}\label{lem:proximal}
Let as before 
$\sigma_q:\widehat{\mathcal{I}}_g^{(q)} \to \GL(\wedge^{g-1}\C^{2g-2})$. 
Then the image of the representation $\sigma_q$ is 
strongly irreducible and proximal. 
\end{lemma}
\begin{proof}
  By Corollary~\ref{zarsiki} the image of $\rho_q$
  intersects $\SU(g-1,g-1)$ in a Zariski dense subgroup. 
  Thus, strong irreducibility of $\sigma_q$ follows from Proposition~\ref{prop:irreducibility}.

  To show that $\sigma_q$ is proximal, it suffices to find 
  an element $A$ in the image of $\widehat{\mathcal{I}}_g^{(q)}$
	whose action on $\wedge^{g-1}\C^{2g-2}$ is proximal.
  Now any matrix $A\in \SU(g-1,g-1)$ as a $(2g-2)\times (2g-2)$-matrix with
	real eigenvalues 
  $\lambda_1>\dots >\lambda_{g-1}>1>\lambda_{g-1}^{-1}>\dots >\lambda_1^{-1}$ 
	will do, and we can find such a matrix in the image of $\widehat{\mathcal{I}}_g^{(q)}$
  since this image is Zariski dense in $\SU(g-1,g-1)$ by Corollary~\ref{zarsiki} again.
\end{proof}

\subsection{Random walks on $\widehat{\mathcal{I}}_g^{(q)}$}~\label{sec:random-walks}

In Proposition~\ref{prop:matrix-condition}, for a choice of a
positive number $\alpha>0$ we formulated
two events for a random walk on $\Ig$ 
and proved that at least one of these events occurs if 
the logarithmic Mahler measure of the Alexander polynomial of the 
hyperbolic three-manifold defined by an element of the walk
vanishes. 
Our end game will be to find a number $\alpha$ so that the probability of any one of
these two events occurring tends to zero with the step-length of the 
random walk. 

To find such a number $\alpha>0$ we use 
results of Benoist and Quint \cite{BQ,BQ2}. 
Let us fix a finitely supported probability measure $\mu$ on $\widehat{\mathcal{I}}_g^{(q)}$
whose support generates $\widehat{\mathcal{I}}_g^{(q)}$ as a semigroup.
Later, this measure will be obtained from a measure on the Torelli group $\Ig$ using the procedure described in Section~\ref{sec:lifting-walks}, but for this section this is not important.
This set up gives rise to a one-sided Bernoulli space $B$ with alphabet $\widehat{\mathcal{I}}_g^{(q)}$
defined by:
\[B = (\widehat{\mathcal{I}}_g^{(q)})^\mathbb{N},\]
The $\sigma$-algebra on this space is generated by \emph{cylinder sets}
\[C(\varphi_1,\ldots,\varphi_n) = \{b\in B\mid\; b_i=\varphi_i\;\forall\;i=1,\ldots,
n\}, \]
for $\varphi_1,\ldots,\varphi_n\in \widehat{\mathcal{I}}_g^{(q)}$. We define a shift invariant
probability measure $\beta$ on $B$ by
\[\beta = \mu^{\otimes \mathbb{N}*}. \]
Observe that by definition we have for a subset $X \subset \widehat{\mathcal{I}}_g^{(q)}$
\[ \mu^{\ast n}(X) = \beta\left( \bigcup_{\substack{\varphi_1, \ldots, \varphi_n\in\widehat{\mathcal{I}}_g^{(q)}\\ \varphi_1\cdots\varphi_n \in X}}C(\varphi_1, \ldots, \varphi_n) \right) \]

The natural Hermitian inner product on $\mathbb{C}^{2g-2}$ for which the basis 
$x_i,y_i$ is orthogonal induces an inner product $(\cdot ,\cdot)$ 
on $V=\wedge^{g-1}\C^{2g-2}$. We denote by $\Vert \cdot \Vert$ the corresponding norm, 
and we let $\Vert A\Vert$ be the operator norm of a
matrix
$A\in \GL(V)$ with respect to this norm. 

We will use the following result by Benoist and Quint, which is a combination of
Theorem 4.28(b) and Theorem 4.31 of \cite{BQ2}.
In its formulation, we use as before the special point 
$e=x_1\wedge\dots \wedge x_{g-1}\in \wedge^{g-1}\mathbb{C}^{2g-2}$.

\begin{theorem}\label{theorem:law-large-numbers}\emph{(The law of large numbers)}
For a fixed finitely supported probability measure $\mu$ on
$\widehat{\mathcal{I}}_g^{(q)}$ and for fixed $q\geq 7$ or for $q=5$, 
there exists a number $\lambda=\lambda(\mu,q)>0$ 
such that for $\beta$-almost every $b\in B$, one has
\[ \frac{1}{n}\log \Vert \rho_q(b_n\cdots b_1)\cdot e\Vert \to \lambda. \]
Furthermore, this convergence also holds in $\mathrm{L}^1(B,\beta)$. That is, the
functions $L_n:B\to\mathbb{R}$ defined by
\[L_n(b) = \frac{1}{n}\log \Vert \rho_q(b_n\cdots b_1)\cdot e\Vert,\]
converge in $\mathrm{L}^1(B,\beta)$ to the constant function $\lambda(\mu,q)$ as $n\to\infty$.
\end{theorem}
\begin{proof}
The results of Benoist and Quint are valid for any random walk on $\mathrm{GL}(V)$ whose
support generates (as a semigroup) a subgroup $\Gamma$ of $\mathrm{GL}(V)$ so that 
the standard representation of $\Gamma$ on $V$ is strongly irreducible and proximal.
That these properties hold true for the representations $\sigma_q$ 
was shown in Lemma  \ref{lem:proximal}. 
\end{proof}

It's important to stress here that 
$\lambda(\mu,q)$ depends only on $\mu$ and $q$, but not on the sample path $b\in B$. 

The following result about random walks on projective spaces is also due to Benoist
and Quint \cite{BQ}.
In its formulation, convolution of measures in $\P(\C^{d})$ is via the orbit map for
the action
of $\GL(\C^{d})$. 
\begin{proposition}\label{thm:bqproj}\cite[Theorem 1.1 (ii)]{BQ}
Let $\mu$ be a measure on $\GL(\C^{d})$ such that the Zariski closure of the
semi-group generated by the support of $\mu$ is semi-simple.
Let $x\in \P(\C^{d})$. Then the convolutions $\mu^{*n}\star \delta_{x}$ converge to
a $\mu$-stationary measure $\nu$ on $\P(\C^{d})$.
\end{proposition}
The following result is due to Gol'dshe\u\i d and Margulis \cite{GM} (see also \cite[Lemma 4.6]{BQ2}).
\begin{proposition}\label{propernomeasure}
Let $\mu$ be a probability measure on $\GL(\C^{d})$. 
Let $\Gamma_{\mu}$ denote the smallest subsemigroup of $\GL(\C^{d})$ 
such that $\mu(\Gamma_{\mu})=1$.
Let $\nu$ be a $\mu$-stationary probability measure on $\P(\C^{d})$.
If $\Gamma_{\mu}$ acts strongly irreducibly on $\C^{d}$ then 
$\nu(\P (W))=0$ for any proper linear subspace $W$ of $\mathbb{C}^d$.
\end{proposition}

\subsection{Controlling Finite Covers}
Recall that $\lambda(\mu,q)>0$ is the Lyapunov exponent associated to a finitely supported
probability measure $\mu$ on $\widehat{\mathcal{I}}_g^{(q)}$ and a representation $\rho_q:\widehat{\mathcal{I}}_g^{(q)}\to \U(g-1,g-1)$ (see Theorem
\ref{theorem:law-large-numbers}). Furthermore, let $\nu = \lim_{n\to\infty} \mu^{\star n}\star \delta_e$
whose existence is guaranteed by Proposition~\ref{thm:bqproj}.

\begin{corollary}\label{cor:probabilities-large-q} 
Let $q\geq 3$< then for every
$\alpha\in (0,\lambda(\mu,q))$ we have:
\[\mu^{*n}\left(\left\{\varphi\in \widehat{\mathcal{I}}_g^{(q)} \mid \vert (\rho_q(\varphi)e,f)\vert <
\exp(\alpha n) \right \} \right) \to 0,\]
as $n\to\infty$.
\end{corollary}

\begin{proof} The claim will follow from two statements, namely:
  \begin{description}
  \item[I] For every $\varepsilon>0$ there exists a
    $\delta=\delta(\varepsilon)>0$ and an $n_0=n_0(\varepsilon)\in\mathbb{N}$ such that:
    \begin{equation*}
      \mu^{*n}\left(\left\{\varphi\in \widehat{\mathcal{I}}_g^{(q)} \mid \frac{\vert
            (\rho_q(\varphi)e,f)\vert}{\Vert\rho_q(\varphi)e\Vert} < \delta \right\}\right) <
      \varepsilon/2,
    \end{equation*}
    for all $n\geq n_0$.
  \item[II] For every $\varepsilon,\delta >0$ there exists an
    $n_1=n_1(\varepsilon,\delta)\in\mathbb{N}$ so that for all $n\geq
    n_1$ we have:
    \begin{equation*}
      \mu^{*n}\left(\left\{\varphi\in\widehat{\mathcal{I}}_g^{(q)}\mid \Vert\rho_q(\varphi)e\Vert <
          \exp(\alpha n)/\delta \right\} \right) < \varepsilon/2, 
    \end{equation*}
  \end{description}
  Assuming these for a moment, the claim is immediate: if $\vert
  (\rho_q(\varphi)e,f)\vert < \exp(\alpha n) $ then
  \[ \frac{\vert
    (\rho_q(\varphi)e,f)\vert}{\Vert\rho_q(\varphi)e\Vert} <
  \delta\;\; \text{or}\;\; \Vert\rho_q(\varphi)e\Vert < \exp(\alpha
  n)/\delta. \] 
  As such, combining \textbf{I} and \textbf{II} proves the corollary.

  \medskip
To prove \textbf{I}, we note that the condition 
$(\cdot,f)=0$ defines a linear subspace 
of $\wedge^{g-1}\C^{2g-2}$ of proper codimension. As such 
\[\nu\left(\left\{v\in \mathbb{P}(\wedge^{g-1}\C^{2g-2})  \mid \frac{\vert (v,f)\vert}{\Vert
v\Vert} =0 \right\} \right) =0\]
by Proposition~\ref{propernomeasure}. By regularity of $\nu$ and continuity of
the function 
\[x\mapsto \frac{\vert (x,f)\vert}{\Vert x \Vert}\]
on $\mathbb{P}(\wedge^{g-1}\C^{2g-2})$,  
there exists a number $\delta>0$ so that
\[\nu\left(\left\{v\in \mathbb{P}(\wedge^{g-1}\C^{2g-2}) 
 \mid \frac{\vert (v,f)\vert}{\Vert v\Vert} \leq  \delta \right\} \right) <
\varepsilon/4. \]
Therefore 
weak convergence of the measures $\mu^{*n}*\delta_e$ to $\nu$ implies
\[ \limsup_{n\to\infty} \mu^{*n}\star\delta_e  \left(\left\{v\in \mathbb{P}(\wedge^{g-1}\C^{2g-2}) 
 \mid \frac{\vert (v,f)\vert}{\Vert v\Vert} \leq \delta \right\} \right) < \varepsilon/4. \]
The inequality we are after for $n$ large enough now follows from the identity
\begin{align} \mu^{*n}\star\delta_e  \left(\left\{v\in \mathbb{P}(\wedge^{g-1}\C^{2g-2}) 
 \mid \frac{\vert (v,f)\vert}{\Vert v\Vert} \leq \delta \right\} \right)\notag\\
= \mu^{*n}\left( \left\{\varphi\in \widehat{\mathcal{I}}_g^{(q)} \mid \frac{\vert
(\rho_q(\varphi)e,f)\vert}{\Vert\rho_q(\varphi)e\Vert} \leq \delta \right\} \right).
\notag\end{align}

\medskip
We will prove Claim \textbf{II} using the law of large numbers (Theorem
\ref{theorem:law-large-numbers}). The $\mathrm{L}^1(B,\beta)$ convergence of the
functions $L_n:B\to\mathbb{R}$ to the constant function $\lambda$ implies that as $n\to\infty$
\[\int_B \vert L_n(b) - \lambda\vert d\beta(b) \to 0. \]
Let $\alpha'>0$ so that $\alpha<\alpha'<\lambda$. The convergence above tells us
that for all $\varepsilon>0$ there exists an $n_1\in\mathbb{N}$ so that for all
$n\geq n_1$:
\[\beta\left(\left\{b\in B\mid\; \frac{1}{n}\log\Vert \rho_q(b_n\cdots b_1)\cdot
e\Vert < \alpha' \right\}\right) < \varepsilon/2. \]
Because the condition on the left is only a condition on the initial $n$ entries of
the infinite path $b\in B$, using the relation between $\beta$ and $\mu$ described at the beginning of the section we have
\[
\beta\left(\left\{b\in B\mid\; L_n(b) \leq \alpha' n \right\}\right) =
\mu^{*n}\left(\left\{\varphi \in \widehat{\mathcal{I}}_g^{(q)}\mid\; \log\Vert \rho_q(\varphi)\cdot
e\Vert \leq \alpha' n \right\}\right). \]
By increasing $n_1$, we can make sure that $\exp(\alpha'n)>\exp(\alpha n)/\delta$.
This shows \textbf{II}.
\end{proof}

\subsection{The proof of the main theorems}
Putting all the above together, we can now prove Theorem~\ref{thm:main-theorem-plain}.
\begin{proof}[The proof of Theorem~\ref{thm:main-theorem-plain}] 
Let us denote by $\mathcal{B}\subset\Ig$ the set of those elements of the Torelli
group that do not have $E_\sigma$ (the condition we defined in the beginning of this section).
For a number $\alpha>0$ (chosen below) let $K_{\alpha, \mu}$ be the finite set given by 
Proposition~\ref{prop:matrix-condition}. The
same proposition shows that for any $q\in\mathbb{N}$ we have
\[ \mathcal{B} \cap \mathrm{supp}(\mu^{\ast n})\subset  \mathcal{A}^{(n)} \cup \bigcup_{k\in K_{\alpha,\mu}} \mathcal{C}_k \]
where
\[\mathcal{A}^{(n)} = \left\{\varphi\in \Ig \mid \vert
(\rho_q(\hat\varphi)e,f)\vert < \exp(\alpha n) \mbox{ for some lift }\hat\varphi\mbox{ of }\varphi\right\}, \]
\[\mathcal{C}_k = \left\{\varphi\in \Ig \mid \det(B_\infty(\varphi))\;\text{contains}\;\Phi_k\;\text{as a factor} )
\right\} \] 
Set $q=3$ and assume that $\alpha<\lambda(\mu,3)$ (in addition to further constraints below).
Consider the projection $\pi:\widehat{\mathcal{I}}_g^{(3)}\to \Ig$. By Lemma~\ref{lem:average}, we have
\[ \mu^{\ast n}(\mathcal{A}^{(n)}) = \zeta^{\ast n}(\pi^{-1}\mathcal{A}^{(n)}) \]
where $\zeta$ is the lifted measure as in Section~\ref{sec:lifting-walks}.
Now, 
\[ \zeta^{\ast n}(\pi^{-1}\mathcal{A}^{(n)}) = \zeta^{\ast n}\left( \left\{\hat\varphi\in \widehat{\mathcal{I}}_g^{(3)} \mid \vert
(\rho_3(\hat\varphi)e,f)\vert < \exp(\alpha n)\right\}\right) \]
and hence as $\alpha<\lambda(\mu,3)$, it follows from  
Corollary~\ref{cor:probabilities-large-q} that 
\[ \mu^{\ast n}(\mathcal{A}^{(n)}) = \zeta^{\ast n}(\pi^{-1}\mathcal{A}^{(n)}) \to 0 \quad (n\to \infty).\]

As $K_{\alpha,\mu}$ is a finite set, for the proof of the theorem it now suffices to 
show that for each fixed $k\in K_{\alpha,\mu}$ we have 
$\mu^{*n}({\mathcal C}_k)\to 0$ as $n\to \infty$. 

Thus let $k\in K_{\alpha,\mu}$ and assume first that $k\geq 3$. 
If $\det(B_\infty(\varphi))$ contains $\Phi_k$ as a factor, then Lemma~\ref{lem:determinant} implies that $(\rho_k(\hat\varphi)e,f)=0$ for a lift $\hat\varphi$ of $\varphi$ to $S_k$. Now the equation 
$(\cdot, f) = 0$ defines a proper linear subspace of
$\wedge^{g-1}\C^{2g-2}$, so arguing as in the proof of
Corollary~\ref{cor:probabilities-large-q} we see that for a finitely supported probability
measure $\zeta$ on $\widehat{\mathcal{I}}_g^{(k)}$ whose support generates  $\widehat{\mathcal{I}}_g^{(k)}$ as a semigroup, we have
that
\[ \zeta^{\ast n}\left( \left\{\tilde\varphi\in \widehat{\mathcal{I}}_g^{(k)} \mid 
(\rho_k(\hat\varphi)e,f) = 0\right\} \right) \to 0 \]
Using again Lemma~\ref{lem:average}, this implies that
\[ \mu^{*n}(\mathcal{C}_{k}) \to 0 \]
for $k\geq 3$.

\smallskip
We are thus left with controlling the case where $\det(B_\infty(\varphi))$
contains at least one of $\Phi_1$, $\Phi_2$ as a
factor, but no other $\Phi_k$ with $k\in K_{\alpha,\mu}$. In other words, we need
to show that 
\[ \mu^{\ast n}\left(\left(\mathcal{C}_1\cup\mathcal{C}_2\right)\setminus \bigcup_{k\in K_{\alpha,\mu}, k\geq 3} \mathcal{C}_k\right)\to 0. \]
To this end, we will decompose the set in question into two sets 
$\mathcal{D}_1,\mathcal{D}_2$
and show that their measures converge to $0$. Namely,
let $\mathcal{D}_i\subset \Ig$ be subset of the Torelli group where
\[\det(B_\infty(\varphi))=\Phi_1^{m_1}\cdot\Phi_2^{m_2}\cdot Q\]
for some polynomial $Q\in\Z[t]$ that does not contain $\Phi_k$ as a
factor for any $k\in\{1,2\}\cup K_{\alpha,\mu}$ and $m_1, m_2\in\mathbb{N}$ 
are such that $m_i\geq m_j\;\forall j\neq
i$ and $m_i\geq 1$. In other words, $\mathcal{D}_i$ is the set where of the two ``problematic'' polynomials $\Phi_1$ and $\Phi_2$, $\Phi_i$ appears with the largest exponent. Note
that indeed
\[ \left(\mathcal{C}_1\cup \mathcal{C}_2\right)\setminus \bigcup_{k\in K_{\alpha,\mu}, k\geq 3} 
\mathcal{C}_k =  \mathcal{D}_1\cup\mathcal{D}_2. \]
Note that there is a number $B>1$ so that
\[ |\Phi_i(\xi)| \leq B, \quad\quad\forall i=1,2\;\forall \xi\in\mathbb{C}, |\xi|=1. \]
Since primitive roots of unity of prime order are dense in the unit
circle, we can choose distinct primes $q_i > 2$ and $q_i$--th roots of unity $\zeta_{q_i}$
so that 
\[ |\Phi_i(\zeta_{q_i})|\leq B^{-1} \quad\quad\forall i=1,2 \]
and therefore
\[ \vert \Phi_i(\zeta_{q_i})\vert \cdot \max\left\{1, \vert
  \Phi_j(\zeta_{q_i})\vert\right\} \leq 1,\]
for $i\neq j\in\{1,2\}$. We will also assume from now on that the map $\iota:\Z[\Z/q_i\Z] \to \C$ introduced in
Section~\ref{subsec:rep} maps the generator $1 \in \Z/q_i\Z$ to
this chosen primitive $q_i$--th root of unity $\zeta_{q_i}$.

We further impose that $\alpha>0$ satisfies
\[ \alpha < \min\{\lambda(\mu,q_1),\;\lambda(\mu,q_2)\}\]

Note that by our choices, using Proposition~\ref{prop:mahler-measure-constraint}, for all $\varphi\in \mathcal{D}_i$ we have that for a lift $\hat\varphi$ 
\[\vert(\rho_{q_i}(\hat\varphi)e,f)\vert=\vert\det(B_\infty(\varphi))(\zeta_{q_i})\vert \leq \exp(\alpha n).\]
Namely,
\[ \vert\det(B_\infty(\varphi))(\zeta_{q_i})\vert = \]\[\vert\Phi_1^{m_1}(\zeta_{q_i})\cdot\Phi_2^{m_2}(\zeta_{q_i})\cdot Q(\zeta_{q_i})| \leq  |Q(\zeta_{q_i})| \]
and arguing as in the proof of Proposition~\ref{prop:matrix-condition} we have
\[|Q(\zeta_{q_i})| \leq \exp(\alpha'\deg(Q)) \leq \exp(\alpha'\deg(\det(B_\infty(\varphi))) <
\exp(\alpha n),\]
where $\alpha'$ is the number computed from $\alpha$ in that proof.

Hence, since $q_i > 2$, we can argue as in the case of the $\mathcal{A}^{(n)}$ to show that
\[ \mu^{\ast n}(\mathcal{D}_i) \to 0 \]

This finishes the proof of the theorem.
\end{proof}

We conclude with the proof of the second main theorem. Namely,
\begin{theorem}\label{thm:dt}
  Let $\mu$ be any measure on $\Ig$ (or a homology
  stabiliser) with finite support so that the support generates $\Ig$ as a semigroup. Fix $d>0$. Then
  \[ \mu^{\ast n}\left(\left\{ \varphi \in \Ig \left| \substack{\displaystyle{b_1(N'_\varphi)>b_1(N_\varphi) \mbox{ for some}}\\\displaystyle{\mbox{ Abelian cover } N'_\varphi \mbox{ of degree } \leq d}}\right.\right\}\right) \to 0\]
as $n\to\infty$.
\end{theorem}

As a first reduction,
note that it suffices to consider cyclic covers. This is due to the fact
that every representation of an Abelian group factors through a cyclic group. See the
discussion in \cite[Section~9.3]{DT} for details on this reduction. As there are only finitely 
many such cyclic covers, it suffices to show the conclusion for one.

Fix a cyclic cover $S_q \to S$. 
Let $\widehat{N_\varphi}$ be the cover of $N_\varphi$ defined by $S_q \to S$.
\begin{proposition}\label{prop:criterion-betti}
  The cover $\widehat{N_\varphi}$ has strictly larger Betti number than
$N_\varphi$ if and only if
  \[ \iota(\det B_q(\varphi)) = 0\]
\end{proposition}
\begin{proof}
  Recall that $E = \ker( H_1(S';\Z) \to H_1(S;\Z) )$ and that 
  $H_1(S';\Z) = \Z[G]^{2g-2} \oplus \Z^2$. We define a map
  \[ \mathcal{J}: H_1(S'; \Z) \to \C^{2g-2} \] 
  by setting it to be zero on
  the trivial summand $\Z^2$ and applying the map $\iota$ (from
  Section~\ref{subsec:rep}) coordinate-wise on $\Z[G]^{2g-2}$. 
  As the kernel of $\iota:\Z[G]\to\C$ is exactly the trivial representation contained
  in $\Z[G]$, the homomorphism $\mathcal{J}$ 
induces an isomorphism of $E$ onto its image
  $\mathcal{J}(E)$. Consequently, 
the induced map $\mathcal{J}_\Q:H_1(S';\Q) \to \C$ induces an 
  isomorphism from $E\otimes \Q$ onto its image $\mathcal{J}_\Q(E\otimes \Q)$

  By Proposition~\ref{prop:describe-cover-homology-nontriv}, we have
  that $b_1(\widehat{N_\varphi}) > b_1(N_\varphi)$ if and only if
  \[ (L'_E + \widehat{\varphi}_*L_E')\otimes\Q \subsetneq E\otimes\Q \]
  is a proper subspace, or equivalently, if
  \[ \iota(L'_E\otimes\Q) + \iota(\widehat{\varphi}_*L_E'\otimes\Q) \subset
\iota(E\otimes \Q) \]
  is a proper subspace. By definition of $B_q$, the latter is the case exactly if 
  $\iota(\det B_q(\varphi)) \neq 0$. This shows the proposition.  
\end{proof}
Arguing as in the proof of Theorem~1, the asymptotic probability that
the determinants $\det B_q(\varphi)$ satisfy the (algebraic) condition
from Proposition~\ref{prop:criterion-betti} will converge to $0$, proving Theorem~2.

\begin{remark}
  In order to extend \cite[Theorem~9.1]{DT} to genus $g \geq 3$ using this method, one
  can argue as follows. As above, it suffices to consider cyclic covers. For a fixed
  cyclic cover, the subgroup $\Gamma$ of the mapping class group which does lift to
the cover and
  so that lifts commute with the deck group action has finite index in the mapping
class group.
  Standard equidistribution results for random walks on finite graphs can be used to
show that
  the desired result is true for a random walk on the mapping class group if it is
true for a 
random walk on $\Gamma$. On $\Gamma$ one can define the representation $\rho$ as
before, and the argument given for Theorem~\ref{thm:dt} applies.
\end{remark}

We end with a sketch of the following result, which addresses a question in \cite[p. 139]{K}.
\begin{theorem}\label{thm:kowalski}
  Let $\mu$ be a probability measure on
  $\mathrm{Mod}(S_g)$ whose finite support generates $\mathrm{Mod}(S_g)$ as a
  semigroup. Then there is an $\alpha > 0$ so that
  \[ \mu^{\ast n}\left( \{\varphi\in\mathrm{Mod}(S_g) \vert \#(H_1(N_\varphi;\mathbb{Z})_{\mathrm{tors}}) < \exp(\alpha n) \right) \to 0 \]
\end{theorem}
\begin{proof}[Sketch of proof]
  Consider the standard representation
  \[ \rho:\mathrm{Mod}(S_g)\to \mathrm{Sp}(2g, \mathrm{Z}) \] on the
  homology of the surface. It is known that this is surjective. As in
  Section~\ref{sec:matrix}, one can define a ``bottom-left block'' $B(\varphi)$ of $\rho(\varphi)$ so that 
  \[ \#(H_1(N_\varphi;\mathbb{Z})_{\mathrm{tors}}) = \det
  B(\varphi) \] supposing that the determinant $\det B(\varphi)$ is
  nonzero. Now, one can use the representation $\rho$ in place of the
  $\rho_q$ to let the mapping class group act on Lagrangian subspaces of $\R^{2g}$. 
  
  Given a basis $v_1,\ldots,v_g\in\R^{2g}$ of a Lagrangian subspace of $\R^{2g}$, we obtain a vector $v_1\wedge\cdots\wedge v_g\in \wedge^g\R^{2g}$. The action of $\mathrm{Sp}(2g, \Z)$ on the subspace $W\subset \wedge^g\R^{2g}$ spanned by all vectors obtained from Lagrangian subspaces of $\R^{2g}$ is known to be irreducible and proximal. As such the same results of Gol'dshe\u\i d and Margulis \cite{GM} and Benoist and Quint \cite{BQ,BQ2} that are used in Section \ref{sec:random-walks} apply.
  
  Thus one can show that, first, with probability converging to $1$, $\det B(\varphi)
  \neq 0$ (as the opposite is a proper subspace of $W$), and in fact it does grow exponentially fast with the length
  $n$ of the walk (by the law of large numbers).
\end{proof}

\vspace{0.5cm}


\noindent{\bf Current affiliations of the authors:} 

\bigskip

\noindent
Hyungryul Baik (hrbaik@kaist.ac.kr)\\
\noindent
DEPARTMENT OF MATHEMATICAL SCIENCES, KAIST  \\
291 DAEHAK-RO, YUSEONG-GU,  DAEJEON,\\
34141, REPUBLIC OF KOREA \\

\vspace{0.5cm}

\noindent
David Bauer (david-bauer@uni-bonn.de)\\
Ursula Hamenst{\"a}dt (ursula@math.uni-bonn.de)\\
Sebastian Hensel (hensel@math.uni-bonn.de)\\
Thorben Kastenholz (tkastenholz@gmx.de)\\
Bram Petri (bpetri@math.uni-bonn.de) \\
Daniel Valenzuela (daniel@valenzuela.de)\\
\noindent
MATHEMATISCHES INSITUT DER UNIVERSIT\"AT BONN\\
ENDENICHER ALLEE 60, 53115 BONN\\
GERMANY

\vspace{0.5cm}

\noindent
Ilya Gekhtman (ilya.gekhtman@yale.edu)

\noindent
DEPARTMENT OF MATHEMATICS, YALE UNIVERSITY\\
10 HILLHOUSE AV, NEW HAVEN, CONNECTICUT 06520\\
USA

\end{document}